\documentclass[12pt]{article}

\usepackage{amssymb,amsmath,amscd}
\usepackage{mathrsfs,mathtools}
\usepackage{amsfonts}
\usepackage{amssymb,amsthm}
\usepackage[arc,all]{xy}
\usepackage{tikz-cd}
\usepackage{hyperref}
\usepackage{stmaryrd}

\DeclareMathOperator{\Tr}{Tr} \DeclareMathOperator{\pt}{pt}

\DeclareMathOperator{\GL}{GL}
\DeclareMathOperator{\gl}{\mathfrak{gl}}
 
 \DeclareMathOperator{\Sym}{Sym}

\DeclareMathOperator{\Ho}{H}

\DeclareMathOperator{\Crit}{Crit} \DeclareMathOperator{\Hom}{Hom}
 
 \DeclareMathOperator{\tr}{tr}

 \DeclareMathOperator{\Surj}{Surj}
\DeclareMathOperator{\coh}{H}
\DeclareMathOperator{\codim}{codim}
\DeclareMathOperator{\TS}{TS}
\DeclareMathOperator{\Sh}{Sh}
\DeclareMathOperator{\Arg}{Arg}

\newcommand{\Z}{{\bf Z}}

\newcommand{\QQ}{\mathbb{Q}}

\newcommand{\Zn}{\mathbb{Z}}
\newcommand{\Nn}{\mathbb{N}}
\newcommand{\Cp}{\mathbb{C}}
\newcommand{\Hc}{\mathcal{H}}
\newcommand{\g}{\mathfrak{g}}
\newcommand{\h}{\mathfrak{h}}
\newcommand{\T}{\mathbb{T}}

\newtheorem{thm}{Theorem}[section]

\newtheorem{prop}[thm]{Proposition}
\newtheorem{example}[thm]{Example}

\newtheorem{defn}[thm]{Definition}

\newtheorem{rem}[thm]{Remark}

\newtheorem*{thm*}{Theorem}

\newcommand{\bcen}{\begin{center}}
\newcommand{\ecen}{\end{center}}

\def\dim{{\rm dim}}
\def\Im{{\rm Im}}

\title{Critical Cohomological Hall Algebra\\and Edge Contraction}

\date{}
\author{Yiqiang Li, Jie Ren}

\begin{document}
\maketitle

\begin{abstract}
We study the behaviors of cohomological Hall algebras and relevant subjects under an edge contraction. Given a quiver with potential and fix an arrow, the edge contraction is a way to construct a new quiver with potential. We show that there is an algebra homomorphism between the cohomological Hall algebras induced by the edge contraction, which preserves the Hopf algebra structure, Drinfeld double, mutation, and the dimensional reduction, and induces relations between scattering diagram and Donaldson-Thomas series.
\end{abstract}

\tableofcontents

\section{Introduction}

Numerous representation-theoretic objects, such as Kac-Moody Lie algebras and Hall algebras, can be associated with a graph. Investigating the behavior of these objects under edge contraction becomes a natural inquiry, and the significance of this exploration stems from its diverse applications, for instance Li-Samer \cite{LS}, Maksimau \cite{Mak}, Riche-Williamson \cite{RW}, with more likely to emerge from further investigation.

The initial examination of this matter was initiated in the paper \cite{Li1} by the first-named author, focusing on the behaviors of quantum groups and Hall algebras. In this article, our attention turns to critical cohomological Hall algebras, referred to as COHAs for conciseness.

COHAs, as defined by Kontsevich and Soibelman in \cite{KoSo3}, serve as cohomological analogues to Ringel's Hall algebras in \cite{Rin}. Simultaneously, they provide a rigorous mathematical framework for the algebras of BPS states in 4-dimensional quantum theories with $N=2$ spacetime supersymmetry. The COHA, denoted as $\mathcal H_{Q, W}$, is associated with an oriented graph $Q$, or quiver, along with a potential $W$, represented as a linear combination of cyclic paths in $Q$.

For an arrow $a_0: i_+\rightarrow i_-$ in $Q$, we denote the edge contraction of $Q$ along $a_0$ as $\widehat Q$, and similarly, the edge contraction of $W$ along $a_0$ as $\widehat W$. In the realm of quiver gauge theory, the edge contraction of a quiver with potential corresponds to the process of Higgsing, as elaborated in Remark 3.2. This article aims to demonstrate that

\begin{thm*}
There is an algebra homomorphism $\mathfrak c: \mathcal H_{Q, W}^= \to \mathcal H_{\widehat Q, \widehat W}$. Here $\mathcal H_{Q, W}^=$ is the subalgebra of $\mathcal H_{Q, W}$ restricted to the dimension vectors $\gamma=(\gamma^i)_{i\in I}$ such that $\gamma^{i_+}=\gamma^{i_-}$, where $I$ is the set of vertices of $Q$. Moreover, this homomorphism is compatible with the Hopf algebra structure on the COHAs.
\end{thm*}

Furthermore, under a mild assumption, the homomorphism extends to a morphism on the Drinfeld doubles of COHAs. It's also compatible with mutations of quivers with potentials. In the specific scenario involving the COHAs of preprojective algebras, a similar homomorphism is established through dimensional reduction.

When the quiver $Q$ is Dynkin, there exists an algebra homomorphism from the Hall algebra of $Q$ to $\mathcal H_Q$ introduced in \cite{Bu}. It is anticipated that the homomorphism $\mathfrak c$ is related to the embeddings of Hall algebras in \cite{Li1}.

Quivers with potentials are in correspondence with 3-dimensional Calabi-Yau (3CY for short) categories, providing a suitable framework for the study of motivic Donaldson-Thomas (DT for short) invariants. The COHA serves as an approach to this theory. We further show that

\begin{thm*}
The edge contraction induces an embedding of the scattering diagrams and a map between the associated DT series.
\end{thm*}

The preprojective algebras gives rise to 2CY categories, which has some interesting invariants, such as Kac Polynomials. It is interesting to relate these invariants under edge contractions.

The relation between COHAs and Yangians is investigated by many authors, see, e.g., \cite{DM,RSYZ,RSYZ1}. The behavior of Yangians under edge contraction is another interesting topic. See e.g. \cite{U}.

The outline of the paper is as follows. In Section 2 we recall the definition of COHA, construction of the Drinfeld double, DT invariants and scattering diagram, dimensional reduction. Then we prove the compatibility of edge contraction with these structures in Section 3.
\\\\
{\it Acknowledgements:}
We would like to thank Ben Davison, Maxim Kontsevich, Yan Soibelman, Yaping Yang and Gufang Zhao for useful communications. J. Ren thanks to IHES for the hospitality.

\section{Cohomological Hall algebra}

We start by recalling the definition of the critical COHA proposed in \cite{KoSo3}. For the convenience of the reader we will closely follow the very detailed exposition from \cite{D1} and \cite{DM} which contains proofs of several statements sketched in \cite{KoSo3} as well as several useful improvements of the loc.cit.

\subsection{Reminder on the equivariant critical cohomology}\label{Reminder}

The definition of the critical COHA in the next section will use the equivariant critical cohomology. We will give an outline of the relevant definitions and properties here, and for more details see \cite{D1}, \cite{DM} and \cite{KoSo3}.

For a complex variety $X$ let $\mathscr{D}_c(X,\QQ)$ denote the category of constructible sheaves on $X$. The category $\mathscr{D}_c(\pt,\QQ)$ contains $\T=\QQ[-2]$ and its square root $\T^{1/2}$. For $\mathcal{F}\in\mathscr D^b_c(X,\QQ)$ and $\mathcal{G}\in\mathscr{D}^b_c(\pt,\QQ)$ we use the abbreviation $\mathcal{F}\otimes\mathcal{G}=\mathcal{F}\otimes(X\rightarrow\pt)^*\mathcal{G}$, and let $\T_X^{d/2}=\QQ_X\otimes\T^{d/2}$ for $d\in\Zn$.

For a smooth equidimensioanl variety $X$ let $\mathcal{IC}_X(\QQ)=\T_X^{-\dim X/2}$. If $X$ is irreducible but not smooth then we define the intersection complex $\mathcal{IC}_X(\QQ)=\T^{-\dim X_{reg}/2}\otimes\widetilde{\mathcal{IC}}_X(\QQ_{X_{reg}})$, where $\widetilde{\mathcal{IC}}_X(\QQ_{X_{reg}})$ is the usual intersection cohomology mixed Hodge module complex of $X$ given by intermediate extension of $\QQ_{X_{reg}}$. Namely, $\widetilde{\mathcal{IC}}_X(\QQ_{X_{reg}})=i_{!*}\QQ_{X_{reg}}=Image(i_!\QQ_{X_{reg}}\rightarrow i_*\QQ_{X_{reg}})$ for the inclusion $i:X_{reg}\rightarrow X$. If $X$ is a finite disjoint union of irreducible varieties then let $\mathcal{IC}_X(\QQ)=\oplus_{Z\in\pi_0(X)}\mathcal{IC}_Z(\QQ)$.

Let $X$ be a complex manifold, and $f:X\rightarrow\mathbb{C}$ a holomorphic function. We define the vanishing cycle functor $\varphi_f$ by $\varphi_{f}\mathcal{F}[-1]:=(R\Gamma_{\{Re(f)\leq0\}}\mathcal{F})_{f^{-1}(0)}$. For any submanifold $X^{sp}\subset X$, the {\it critical cohomology with compact support} $\coh_c^{\bullet,crit}(X^{sp},f)$ is defined to be $\coh_c^\bullet(X^{sp},\varphi_f\QQ_X)$.

From now on we will abbreviate $R\mathscr{F}$ to $\mathscr{F}$ for any functor $\mathscr{F}$.

If $j:X'\rightarrow X$ is a closed embedding, then
\begin{equation}\label{PF}
\varphi_fj_*\rightarrow j_*\varphi_{fj}
\end{equation}
is a natural isomorphism of functors.

For any smooth $j$ the natural transformation
\begin{equation}\label{PB}
j^*\varphi_f\rightarrow\varphi_{fj}j^*
\end{equation}
is an isomorphism.

\begin{rem}
Indeed the critical COHA is a $\mathbb{Z}_{\geqslant0}^I$-graded algebra in the category $\mathscr{D}(\textbf{MMHS})$, where $\textbf{MMHS}$ denotes the category of monodromic mixed Hodge structures. More generally we have the category of monodromic mixed Hodge modules $\textbf{MMHM}(X)$ on $X$ and $\textbf{MMHS}\cong\textbf{MMHM}(\pt)$. There is an exact functor $\mathscr D(\textbf{MMHM}(X))\rightarrow\mathscr D_c(X,\QQ)$ and the restricted functor $\textbf{MMHM}(X)\rightarrow Perv(X)$ is faithful, where $Perv(X)$ is the category of perverse sheaves on $X$. Interested readers please refer to \cite{D1}, \cite{DM} and \cite{KoSo3} for details.
\end{rem}

Next let's define the equivariant critical cohomology with compact support. Assume that $X$ is a smooth algebraic manifold carrying a $G$ action, where $G$ is a linear algebraic group. We also assume that $G$ is special, i.e., all étale locally trivial principal $G$-bundles are Zariski locally trivial. Such $G$ is connected (see \cite{Jo1}, \cite{C}). Let $V_1\subset V_2\subset\cdot\cdot\cdot$ be a chain of $G$-representations, and $U_1\subset U_2\subset\cdot\cdot\cdot$ be an ascending chain of subvarieties of the underlying vector spaces of $V_k$, considered as $G$-equivariant algebraic varieties. We assume that $G$ acts scheme-theoretically freely on each $U_k$, that each principal bundle quotient $U_k\rightarrow U_k/G$ exists in the category of schemes, and that $\lim\limits_{k\rightarrow\infty}\codim_{V_k}(V_k\setminus U_k)=\infty$. The map $X\times U_k\rightarrow X\times_GU_k$ exists as a principal bundle quotient in the category of schemes by \cite{EG}. Consider a $G$-invariant morphism $f:X\rightarrow\mathbb{A}^{1}$ and denote $f_k:X\times_GU_k\rightarrow\mathbb{A}^1$ the induced map.

We fix an embedding $G\subset\GL_n(\mathbb{C})$ for some $n$, and let $V_k=\Hom(\mathbb{C}^k,\mathbb{C}^n)$ and $U_k=\Surj(\mathbb{C}^k,\mathbb{C}^n)$ where $\Surj(\mathbb{C}^{k}, \mathbb{C}^{n})$ means the set of the surjective maps from $\mathbb{C}^{k}$ to $\mathbb{C}^{n}$. If $k\geqslant n$ then $U_k$ carries a free $G$-action via the $\GL_n(\mathbb{C})$-action on $\mathbb{C}^n$. Let $v:X^{sp}\rightarrow X$ be an inclusion of submanifold, and $v_k:X^{sp}\times_GU_k\rightarrow X\times_GU_k$ the induced map. Then the {\it equivariant critical cohomology with compact support} on $X^{sp}$ is defined as $$\coh_{c,G}^{\bullet,crit}(X^{sp},f)=\lim\limits_{k\rightarrow\infty}\coh_c^\bullet(X^{sp}\times_GU_k,v_k^*\varphi_{f_k}\T_{X_k}^{-\dim U_k}).$$
It's endowed with a $\coh_G(\pt,\QQ)$-module and a $\coh_G(X,\QQ)$-module structure.

Let $X_1$ and $X_2$ be complex algebraic manifolds acted on by special complex algebraic groups $G_1$ and $G_2$ respectively, then we have the following morphism in $\mathscr D_c^b((X_1\times_{G_1}U_k)\times(X_2\times_{G_2}U_k),\QQ)$: \begin{equation}\label{TSsheaves}
\begin{array}{ll} &R\Gamma_{\{Re(f_{1,k})\leqslant 0\}}\mathbb{Q}_{X_1\times_{G_1}U_k}\boxtimes R\Gamma_{\{Re(f_{2,N})\leqslant 0\}}\mathbb{Q}_{X_2\times_{G_2}U_k}\\\rightarrow &R\Gamma_{\{Re(f_{1,k}\boxplus f_{2,k})\leqslant 0\}}\mathbb{Q}_{(X_1\times_{G_1}U_k)\times(X_2\times_{G_2}U_k)} \end{array}
\end{equation}
which induces the following Thom-Sebastiani isomorphism:
\begin{equation}
\coh_{c,G_1}^{\bullet,crit}((X_1)_0,f_1)\otimes\coh_{c,G_2}^{\bullet,crit}((X_2)_0,f_2)\rightarrow\coh_{c,G_1\times G_2}^{\bullet,crit}((X_1)_0\times(X_2)_0,f_1\boxplus f_2),
\end{equation}
where $(X_i)_0=f_i^{-1}(0)$, and $f_1\boxplus f_2=\pi_1^*f_1+\pi_2^*f_2$ for projections $\pi_i:X_1\times X_2\rightarrow X_i$. Assume that the critical locus of $f$ is contained in $f^{-1}(0)$, then the Thom-Sebastiani isomorphism implies the natural isomorphism
\begin{equation}\label{TS}
\coh_{c,G_1}^{\bullet,crit}(X_1,f_1)\otimes\coh_{c,G_2}^{\bullet,crit}(X_2,f_2)\rightarrow\coh_{c,G_1\times G_2}^{\bullet,crit}(X_1\times X_2,f_1\boxplus f_2).
\end{equation}

Let $g:X\rightarrow Y$ be a $G$-equivariant morphism of complex algebraic manifolds, and $f$ be a $G$-invariant holomorphic function on $Y$, and $Y^{sp}\subset Y$ be a $G$-invariant submanifold and $X^{sp}=g^{-1}(Y^{sp})$. Then we can define the pullback map
\begin{equation}
g^*: \coh_{c,G}^{\bullet,crit}(Y^{sp},f)^\vee\rightarrow\coh_{c,G}^{\bullet,crit}(X^{sp},fg)^\vee\otimes\T^{\dim X-\dim Y}.
\end{equation}
If $g$ is an affine fibration, then there is a pushforward map
\begin{equation}
g_*: \coh_{c,G}^{\bullet,crit}(X^{sp},fg)^\vee\rightarrow\coh_{c,G}^{\bullet,crit}(Y^{sp},f)^\vee[\mathfrak{eu}(g)^{-1}],
\end{equation}
where $\mathfrak{eu}(g)\in\coh_G(X,\QQ)\cong\coh_G(Y,\QQ)$ is the Euler characteristic of $g$ defined in \cite[Sec.~2.7]{D1}.

For $g$ proper, we can define the pushforward map
\begin{equation}
g_*: \coh_{c,G}^{\bullet,crit}(X^{sp},fg)^\vee\rightarrow\coh_{c,G}^{\bullet,crit}(Y^{sp},f)^\vee.
\end{equation}
Please see \cite{D1} for details.

If there's no $G$-action then the above $\coh(X,\QQ)$-module structure, $\mathfrak{eu}(g)$, pullback and pushforward maps are also defined between the (non-equivariant) critical cohomologies with compact support.

\begin{thm}\label{base}{\rm(see \cite[Prop.~2.14]{D1})}
Let $$\begin{xy}
(0,20)*+{X}="v1";
(20,20)*+{X'}="v2";
(0,0)*+{Y}="v3";
(20,0)*+{Y'}="v4";
{\ar@{->}^{g} "v1";"v2"};
{\ar@{->}^{h} "v3";"v4"};{\ar@{->}^{j} "v1";"v3"}; {\ar@{->}^{j'} "v2";"v4"}
\end{xy}$$ be a Cartesian diagram of complex connected algebraic manifolds in which g and h are either proper, or affine fibrations. Assume there is a natural isomorphism $\nu:g^*j'^!\rightarrow j^!h^*$ such that the diagrams $$\begin{xy}
(0,20)*+{j^*h^*\{reldim(j)\}}="v1";
(30,20)*+{j^!h^*}="v2";
(30,0)*+{g^*j'^!}="v4";
{\ar@{->} "v1";"v2"};
{\ar@{->}^{\nu} "v4";"v2"}; {\ar@{->} "v1";"v4"}
\end{xy}$$ and $$\begin{xy}
(0,20)*+{h^*j'_!j'^!}="v1";
(30,20)*+{j_!j^!h^*}="v2";
(30,0)*+{h^*}="v4";
{\ar@{->} "v1";"v2"};
{\ar@{->} "v2";"v4"}; {\ar@{->} "v1";"v4"}
\end{xy}$$ commute, where reldim(j) is the relative dimension of j. Let $f$ be a holomorphic function on $Y'$. Let $Y'^{sp}\subset Y'$ be a subvariety, and define $Y^{sp}=h^{-1}(Y'^{sp})$, $X'^{sp}=j'^{-1}(Y'^{sp})$ and $X^{sp}=j^{-1}(Y^{sp})$. If $h$ and $g$ are proper then the following diagram commutes: $$\begin{xy}
(0,20)*+{\coh_c^{\bullet,crit}(X^{sp},fj'g)^\vee}="v1";
(50,20)*+{\coh_c^{\bullet,crit}(X'^{sp},fj')^\vee}="v2";
(0,0)*+{\coh_c^{\bullet,crit}(Y^{sp},fh)^\vee}="v3";
(50,0)*+{\coh_c^{\bullet,crit}(Y'^{sp},f)^\vee}="v4";
{\ar@{->}^{g_*} "v1";"v2"};
{\ar@{->}^{h_*} "v3";"v4"};{\ar@{->}^{j^*} "v3";"v1"}; {\ar@{->}^{j'^*} "v4";"v2"}
\end{xy}$$ If h and g are affine fibrations and we assume invertibility of the corresponding Euler classes, then the following diagram commutes: $$\begin{xy}
(0,20)*+{\coh_c^{\bullet,crit}(X^{sp},fj'g)^\vee}="v1";
(50,20)*+{\coh_c^{\bullet,crit}(X'^{sp},fj')^\vee[\mathfrak{eu}(g)^{-1}]}="v2";
(0,0)*+{\coh_c^{\bullet,crit}(Y^{sp},fh)^\vee}="v3";
(50,0)*+{\coh_c^{\bullet,crit}(Y'^{sp},f)^\vee[\mathfrak{eu}(h)^{-1}]}="v4";
{\ar@{->}^{g_*} "v1";"v2"};
{\ar@{->}^{h_*} "v3";"v4"};{\ar@{->}^{j^*} "v3";"v1"}; {\ar@{->}^{j'^*} "v4";"v2"}
\end{xy}$$
\end{thm}

\begin{thm}{\rm(\cite[Prop.~2.16]{D1})}
Let $g:X\rightarrow Y$ be a $G$-equivariant map of manifolds. Then
\begin{itemize}
\item[1)] If $g$ is a closed embedding, then $$g^*g_*:\coh_{c,G}^{\bullet,crit}(X^{sp},fg)^\vee\rightarrow\coh_{c,G}^{\bullet,crit}(X^{sp},fg)^\vee$$ is given by multiplication by Euler class of the normal bundle $\mathbf{N}_{X/Y}$.
\item[2)] If $g$ is an affine fibration and the Euler class $\mathfrak{eu}(g)$ is a non zero divisor in $\coh_G(X,\QQ)$, then the map $$g^*g_*:\coh_{c,G}^{\bullet,crit}(X^{sp},fg)^\vee\rightarrow\coh_{c,G}^{\bullet,crit}(X^{sp},fg)^\vee[\mathfrak{eu}(g)^{-1}]$$ is given by division by $\mathfrak{eu}(g)$.
\end{itemize}
\end{thm}

\subsection{COHA of a quiver with potential}\label{COHA}

Given a quiver with potential, or more generally a smooth algebra with potential, Kontsevich-Soibelman defined a cohomological Hall algebra in \cite{KoSo3}. Explicitly, for a quiver with potential $(Q,W)$, where $W\in\mathbb CQ/[\mathbb CQ,\mathbb CQ]$, we denote the set of vertices by $I$, the set of arrows by $\Omega$, and $a_{ij}\in\mathbb{Z}_{\geqslant0}$ the number of arrows from $i$ to $j$ for $i, j\in I$. The sources and targets of arrows are maps $s,t:\Omega\rightarrow I$. Fix a dimension vector $\gamma=(\gamma^i)_{i\in I}$, let $\mathbf{M}_\gamma$ be the space of $Q$-representations of dimension $\gamma$, which is acted on by $\mathbf{G}_\gamma:=\prod_{i\in I}\GL(\gamma^i,\Cp)$. For a fixed $Q$-representations $M\in\mathbf M_\gamma$, the linear map associated with an arrow $a$ is denoted by $M_a$. The potential $W$ defines a $\mathbf{G}_\gamma$-invariant function $\Tr W_\gamma$ on $\mathbf{M}_\gamma$, which is denoted by $W_\gamma$ for simplicity.

Let $\gamma_1$, $\gamma_2$ be a pair of dimensions such that $\gamma_1+\gamma_2=\gamma$. Denote $\mathbf{M}_{\gamma_1,\gamma_2}\subset\mathbf{M}_\gamma$ the subspace of $Q$-representations with $\gamma_1$ dimensional subrepresentations, and $W_{\gamma_1,\gamma_2}$ the restriction of $W_\gamma$ on $\mathbf M_{\gamma_1,\gamma_2}$. The group $\mathbf{G}_{\gamma_1,\gamma_2}\subset\mathbf{G}_\gamma$ consisting of transformations preserving subspaces $(\Cp^{\gamma^i_1}\subset\Cp^{\gamma^i})_{i\in I}$ acts on $\mathbf{M}_{\gamma_1,\gamma_2}$.

Suppose that we are given $\mathbf{G}_\gamma$-invariant subsets $\mathbf{M}_\gamma^{sp}\subset\mathbf{M}_\gamma$ for all $\gamma\in\mathbb{Z}_{\geqslant0}^I$ satisfying the following conditions:
\begin{itemize}
\item[$\bullet$]$\mathbf{M}_\gamma^{sp}\cap\Crit(\Tr W_\gamma)\subset(\Tr W_\gamma)^{-1}(0)$ where $\Crit(\Tr W_\gamma)$ is the critical locus of $\Tr W_\gamma$,
\item[$\bullet$]for any short exact sequence
$0\rightarrow E_{1}\rightarrow E\rightarrow E_{2}\rightarrow0$ of
representations of $Q$ with dimension vectors
$\gamma_{1}, \gamma:=\gamma_{1}+\gamma_{2}, \gamma_{2}$
respectively, $E\in\textbf{M}_\gamma^{sp}$ if and
only if $E_{1}\in\textbf{M}_{\gamma_{1}}^{sp}$, and
$E_{2}\in\textbf{M}_{\gamma_{2}}^{sp}$.
\end{itemize}

Let $\chi_{Q}(\gamma_{1},\gamma_{2})=-\sum\limits_{i,j\in
I}a_{ij}\gamma_{1}^{i}\gamma_{2}^{j}+\sum\limits_{i\in
I}\gamma_{1}^{i}\gamma_{2}^{i}$ be the Euler form on the $K_{0}$ group of the category of finite dimensional representations of $Q$. Its antisymmetrization is $\langle\gamma_1,\gamma_2\rangle=\chi_Q(\gamma_1,\gamma_2)-\chi_Q(\gamma_2,\gamma_1)$.

Fix a framing vector $\omega=(\omega^i)_{i\in I}$, we form a new quiver $Q_f=(I\sqcup\{\infty\}, \Omega\sqcup\{a_{i,l_i}: \infty\rightarrow i| i\in I, 1\leqslant l_i\leqslant\omega^i\})$. For a dimension vector $\gamma=(\gamma^i)_{i\in I}$, the space of the representations of the framed quiver with dimension $(\gamma, 1)$ is denoted by $\mathbf{M}_{\gamma,\omega}$, namely, $\mathbf{M}_{\gamma,\omega}\cong\mathbf{M}_\gamma\times\prod_{i\in I}\Hom(\mathbb{C}^{\omega^i}, \mathbb{C}^{\gamma^i})$. It contains an open subset $\mathbf{M}_{\gamma,\omega}^\circ=\mathbf{M}_\gamma\times\prod_{i\in I}\Surj(\mathbb{C}^{\omega^i}, \mathbb{C}^{\gamma^i})$. Here $\prod_{i\in I}\Hom(\mathbb{C}^{\omega^i}, \mathbb{C}^{\gamma^i})$ and $\prod_{i\in I}\Surj(\mathbb{C}^{\omega^i}, \mathbb{C}^{\gamma^i})$ play the respective roles of $V_k$ and $U_k$ in the definition of equivariant critical cohomology with compact support in section \ref{Reminder}. Furthermore $\mathbf{M}_{\gamma,\omega}^\circ$ contains an open subset $\mathbf{M}_{\gamma,\omega}^{\circ,\heartsuit}=\mathbf{M}_\gamma^\heartsuit\times\prod_{i\in I}\Surj(\mathbb{C}^{\omega^i}, \mathbb{C}^{\gamma^i})$, where $\mathbf{M}_\gamma^\heartsuit=\{\rho\in\mathbf{M}_\gamma|M_{a_0}\;is\;an\;isomorphism\}$ for a fixed arrow $a_0$. The group $\mathbf{G}_\gamma$ acts on $\mathbf{M}_{\gamma,\omega}^\circ$ and $\mathbf{M}_{\gamma,\omega}^{\circ,\heartsuit}$ freely, and the quotient spaces are denoted by $\mathcal{M}_{\gamma,\omega}^\circ$ and $\mathcal{M}_{\gamma,\omega}^{\circ,\heartsuit}$. We add the superscript $sp$ to the above notations if we consider the corresponding spaces replacing $\mathbf M_\gamma$ by $\mathbf{M}_\gamma^{sp}$.

Now let's consider the following maps: the affine fibrations
\begin{equation*}
p_{\gamma_1,\gamma_2,\omega_1,\omega_2}^\circ:\mathbf{M}_{\gamma_1,\gamma_2;\omega_1,\omega_2}^\circ/\mathbf{G}_{\gamma_1}\times\mathbf{G}_{\gamma_2}\rightarrow(\mathbf{M}_{\gamma_1,\omega_1}^\circ\times\mathbf{M}_{\gamma_2,\omega_2}^\circ)/\mathbf{G}_{\gamma_1}\times\mathbf{G}_{\gamma_2}
\end{equation*}
and
\begin{equation*}
q_{\gamma_1,\gamma_2,\omega_1,\omega_2}^\circ:\mathbf{M}_{\gamma_1,\gamma_2;\omega_1,\omega_2}^\circ/\mathbf{G}_{\gamma_1}\times\mathbf{G}_{\gamma_2}\rightarrow\mathbf{M}_{\gamma_1,\gamma_2;\omega_1,\omega_2}^\circ/\mathbf{G}_{\gamma_1,\gamma_2},
\end{equation*}
the inclusion
\begin{equation*}
j_{\gamma_1,\gamma_2,\omega_1,\omega_2}^\circ:\mathbf{M}_{\gamma_1,\gamma_2;\omega_1,\omega_2}^\circ/\mathbf{G}_{\gamma_1,\gamma_2}\rightarrow\mathbf{M}_{\gamma,\omega}^\circ/\mathbf{G}_{\gamma_1,\gamma_2},
\end{equation*}
and the proper map
\begin{equation*}
pr_{\gamma_1,\gamma_2,\omega_1,\omega_2}^\circ:\mathbf{M}_{\gamma,\omega}^\circ/\mathbf{G}_{\gamma_1,\gamma_2}\rightarrow\mathbf{M}_{\gamma,\omega}^\circ/\mathbf{G}_\gamma.
\end{equation*}
Let $\Hc_\gamma=\coh_{c,\mathbf{G}_\gamma}^{\bullet,crit}(\mathbf{M}_\gamma^{sp},W_\gamma)^\vee\otimes\T^{\dim\mathbf{M}_\gamma/\mathbf{G}_\gamma}$, where $\dim\mathbf M_\gamma/\mathbf G_\gamma=-\chi_Q(\gamma,\gamma)$, and $\Hc=\oplus_{\gamma\in\mathbb{Z}_{\geqslant0}^I}\Hc_\gamma$. Denote
$\alpha=p_{\gamma_1,\gamma_2,\omega_1,\omega_2}^{\circ,*}$, $\beta=q_{\gamma_1,\gamma_2,\omega_1,\omega_2}^{\circ,*}$, $\zeta=j_{\gamma_1,\gamma_2,\omega_1,\omega_2,*}^\circ$, $\delta=pr_{\gamma_1,\gamma_2,\omega_1,\omega_2,*}^\circ$, and
\begin{equation}
\begin{array}{ll}
m_{\gamma_1,\gamma_2}=\delta\zeta\beta^{-1}\alpha\TS:&\coh_{c,\mathbf{G}_{\gamma_1}}^{\bullet,crit}(\mathbf{M}_{\gamma_1}^{sp},W_{\gamma_1})^\vee\otimes\coh_{c,\mathbf{G}_{\gamma_2}}^{\bullet,crit}(\mathbf{M}_{\gamma_2}^{sp},W_{\gamma_2})^\vee\\&\rightarrow\coh_{c,\mathbf{G}_\gamma}^{\bullet,crit}(\mathbf{M}_\gamma^{sp},W_\gamma)^\vee\otimes\T^{-\chi_Q(\gamma_2,\gamma_1)},
\end{array}\nonumber
\end{equation}
where TS is the Thom-Sebastiani isomorphism. Define the multiplication $m=\sum_{\gamma_1,\gamma_2}m_{\gamma_1,\gamma_2}:\Hc\otimes\Hc\rightarrow\Hc$, then ($\Hc,m$) is the {\it critical cohomological Hall algebra} associated to $(Q,W)$. Denote $\mathcal{H}^s$ the spherical COHA, which is defined to be the subalgebra of $\Hc$ generated by $\mathcal{H}_{e_i}$ for all $i\in I$.

\begin{rem}
The critical COHA in \cite{D1,DM} is defined using a different twist: $\Hc_\gamma=\coh_{c,\mathbf{G}_\gamma}^{\bullet,crit}(\mathbf{M}_\gamma^{sp},W_\gamma)^\vee\otimes\T^{\frac{1}{2}\dim\mathbf{M}_\gamma/\mathbf{G}_\gamma}$, and
$$\bigoplus_{\gamma_1\in\mathbb{Z}_{\geqslant0}^I}\mathcal{F}_{\gamma_1}\boxtimes_+^{tw}\bigoplus_{\gamma_2\in\mathbb{Z}_{\geqslant0}^I}\mathcal{G}_{\gamma_2}:=\bigoplus_{\gamma\in\mathbb{Z}_{\geqslant0}^I}(\bigoplus_{\gamma_1+\gamma_2=\gamma}\mathcal{F}_{\gamma_1}\otimes\mathcal{G}_{\gamma_2})\otimes\T^{\chi_Q(\gamma_1,\gamma_2)/2-\chi_Q(\gamma_2,\gamma_1)/2},$$
where $\mathcal{F}_{\gamma_1}$ and $\mathcal{G}_{\gamma_2}$ are objects in $\mathscr{D}^{b}(\textbf{MMHS})$.
\end{rem}

There is a localised comultiplication on $\Hc$ (see \cite{D1} for details). Let $\overleftarrow{\alpha}=p_{\gamma_1,\gamma_2,\omega_1,\omega_2,*}^\circ$, $\overleftarrow{\beta}=q_{\gamma_1,\gamma_2,\omega_1,\omega_2,*}^\circ$,
$\overleftarrow{\zeta}=j_{\gamma_1,\gamma_2,\omega_1,\omega_2}^{\circ,*}$, and
$\overleftarrow{\delta}=pr_{\gamma_1,\gamma_2,\omega_1,\omega_2}^{\circ,*}$. Then the comultiplication is defined as $\Delta=\sum_{\gamma_1,\gamma_2}\Delta_{\gamma_1,\gamma_2}$, where
\begin{equation}
\begin{array}{ll}
 \Delta_{\gamma_1,\gamma_2}=\TS\overleftarrow{\alpha}\overleftarrow{\beta}\overleftarrow{\zeta}\overleftarrow{\delta}:&\coh_{c,\mathbf{G}_\gamma}^{\bullet,crit}(\mathbf{M}_\gamma^{sp},W_\gamma)^\vee\rightarrow\\&\coh_{c,\mathbf{G}_{\gamma_1}}^{\bullet,crit}(\mathbf{M}_{\gamma_1}^{sp},W_{\gamma_1})^\vee\otimes\coh_{c,\mathbf{G}_{\gamma_2}}^{\bullet,crit}(\mathbf{M}_{\gamma_2}^{sp},W_{\gamma_2})^\vee\otimes\T^{\chi_Q(\gamma_1,\gamma_2)}[L_{\gamma_1,\gamma_2}^{-1}],
\end{array}\nonumber
\end{equation}
here $L_{\gamma_1,\gamma_2}=\prod_{i,j\in I}\prod_{\alpha_1=1}^{\gamma_1^i}\prod_{\alpha_2=\gamma_1^j+1}^{\gamma_1^j+\gamma_2^j}(x_{j,\alpha_2}-x_{i,\alpha_1})\in\coh_{\mathbf{G}_{\gamma_1}}(\pt)\otimes\coh_{\mathbf{G}_{\gamma_2}}(\pt)$, and $\coh_{\mathbf G_\gamma}(\pt)\cong\QQ[(x_{i,\alpha})_{i\in I,\alpha\in\{1,\ldots,\gamma^i\}}]^{\Sym_\gamma}$ for $\Sym_\gamma=\prod_{i\in I}\Sym_{\gamma^i}$.

One can choose $\mathbf{M}_\gamma^{sp}$ to be the space of $Z$-semistable representations $\mathbf M_\gamma^{Z\text -ss}$ for a fixed central charge (a.k.a. stability function) $Z:\Zn^I\rightarrow\mathbb H_+:=\{z\in\Cp\mid\Im z>0\}\subset\Cp$, where an object $E$ is called semistable if any subobject $F\subset E$ satisfies $\Arg(F):=\Arg(Z(\dim F))\leqslant\Arg(E):=\Arg(Z(\dim E))$, where $\Arg(z)$ is the argument of a complex number $z$. Moreover, an object $E$ is stable if any proper subobject $F$ satisfies $\Arg(F)<\Arg(E)$.

\subsection{Assumption of shuffle description}

We assume that COHA has a shuffle description (see \cite{RSYZ}). For the proof of existence of shuffle description of K-theoretic Hall algebras see \cite{P}, and for COHA of preprojective algebras see \cite{YaZha1}), namely, there is a $\mathbb{Z}_{\geqslant0}^I$-graded algebra homomorphism $\Hc\rightarrow\mathbf{H}$ from the COHA to the shuffle algebra, which is an isomorphism after passing to the localization $-\otimes_{\QQ[\mathfrak{h}_\gamma]^{\Sym_\gamma}}\QQ(\mathfrak{h}_\gamma)^{\Sym_\gamma}$ for each $\gamma$. The induced algebra epimorphism $\Hc^s\rightarrow\mathbf{H}^s$ is an isomorphism after passing to the same localization. Here $\mathbf{H}$ is defined as follows. For $\gamma\in\mathbb{Z}_{\geqslant0}^I$, let $\g=\prod_{i\in I}\gl_{\gamma^i}$ be the Lie algebra of $\mathbf{G}_\gamma$, and $\h_\gamma\subset\g_\gamma$ be its Cartan subalgebra. The group $\Sym_\gamma=\prod_{i\in I}\Sym_{\gamma^i}$ acts on $\QQ[\h_\gamma]$, which is a polynomial ring in variables $x_{[1,\gamma]}\coloneqq(x_{i,\alpha})_{i\in I,\alpha\in\{1,\ldots,\gamma^i\}}$. Let $\mathbf{H}=\oplus_{\gamma\in\mathbb{Z}_{\geqslant0}^I}\QQ[\h_\gamma]^{\Sym_\gamma}$, and $\Sh(\gamma_1,\gamma_2)\subset\Sym_{\gamma_1,\gamma_2}$ be the set of shuffles of $(\gamma_1,\gamma_2)$ (namely, the elements preserving the relative order of $\{1,\ldots,\gamma_1^i\}$ and $\{\gamma_1^i+1,\ldots,\gamma_1^i+\gamma_2^i\}$). The multiplication on $\mathbf{H}$ is given by
\begin{equation}\label{shufflemult}
\begin{array}{ll}
\QQ[\h_{\gamma_1}]^{\Sym_{\gamma_1}}\otimes\QQ[\h_{\gamma_2}]^{\Sym_{\gamma_2}}\rightarrow\QQ[\h_{\gamma_1+\gamma_2}]^{\Sym_{\gamma_1+\gamma_2}},\\\\
f_1*f_2=\sum\limits_{\sigma\in\Sh(\gamma_1,\gamma_2)}\sigma(f_1(x_{[1,\gamma_1]})f_2(x_{[\gamma_1,\gamma_1+\gamma_2]})fac(x_{[1,\gamma_1]}|x_{[\gamma_1,\gamma_1+\gamma_2]}),
\end{array}
\end{equation}
where $fac(x_{[1,\gamma_1]}|x_{[\gamma_1,\gamma_1+\gamma_2]})=\frac{\prod_{i,j\in I}\prod_{\alpha_1=1}^{\gamma_1^i}\prod_{\alpha_2=\gamma_1^j+1}^{\gamma_1^j+\gamma_2^j}(x_{j,\alpha_2}-x_{i,\alpha_1})^{a_{ij}}}{\prod_{i\in I}\prod_{\alpha_1=1}^{\gamma_1^i}\prod_{\alpha_2=\gamma_1^i+1}^{\gamma_1^i+\gamma_2^i}(x_{i,\alpha_2}-x_{i,\alpha_1})}$.

Let $\mathbf{H}^s$ be the spherical shuffle algebra, which is the subalgebra of $\mathbf{H}$ generated by $\mathbf{H}_{e_i}$ for all $i\in I$.

\subsection{Drinfeld double}

To construct the Drinfeld double of COHA, we need the following Hpof structures from \cite{RSYZ}.

Let $H^0 := \Cp[\psi_{i,s}|i\in I, s\in\mathbb{N}]$, and $\psi_i(z)= 1+\sum_{s\geqslant 0}\psi_{i,s}z^{-s-1}\in H^0[\![z^{-1}]\!]$ the generating series of $\psi_{i,s}$'s. Then the extended shuffle algebra $\mathbf{H}^{\mathfrak{e}}=H^0\ltimes\mathbf{H}$ is defined via the $H^0$-action on $\mathbf{H}$ by $\psi_if\psi_i^{-1}=f\frac{fac(z|x_{[1,\gamma]})}{fac(x_{[1,\gamma]}|z)}$ for any $f\in\mathbf{H}_\gamma$. Let $(\mathbf{H}_{\gamma_1}^{\mathfrak{e}}\otimes\mathbf{H}_{\gamma_2}^{\mathfrak{e}})_{loc}$ be the localization away from the union of the null divisors of $fac(x_{[1,\gamma_1]}|x_{[\gamma_1+1,\gamma]})$ and $\psi_i(z)$, then a localized comultiplication $\Delta:\mathbf{H}^{\mathfrak{e}}_{\gamma}\rightarrow\sum_{\gamma_1+\gamma_2=\gamma}(\mathbf{H}^{\mathfrak{e}}_{\gamma_1}\otimes\mathbf{H}^{\mathfrak{e}}_{\gamma_2})_{loc}$ is defined as follows: $$\Delta(\psi_i(z))=\psi_i(z)\otimes\psi_i(z), i\in I,$$ $$\Delta(f(x_{[1,\gamma]}))=\sum_{\gamma_1+\gamma_2=\gamma}\frac{\psi_{[1,\gamma_1]}(x_{[\gamma_1+1,\gamma]})f(x_{[1,\gamma_1]}\otimes x_{[\gamma_1+1,\gamma]})}{fac(x_{[\gamma_1+1,\gamma]}|x_{[1,\gamma_1]})},f(x_{[1,\gamma]})\in\mathbf H^{\mathfrak k}_\gamma,$$
in particular,
$$\Delta(f(x_i)=\psi_i(x_i)\otimes f(x_i)+f(x_i)\otimes 1\in\mathbf{H}^{\mathfrak{e}}_{e_i}\otimes\mathbf{H}^{\mathfrak{e}}_{e_i}, f(x_i)\in\mathbf{H}^{\mathfrak{e}}_{e_i}.$$ Define the counit $$\varepsilon:\mathbf{H}^{\mathfrak{e}}\rightarrow\Cp, \quad\psi_i(z)\mapsto 1, \quad f(x_{[1,\gamma]})\mapsto 0,$$ and the antipode
\begin{equation}
\begin{array}{ll}
S:&\mathbf{H}^{\mathfrak{e}}_{loc}\rightarrow\mathbf{H}^{\mathfrak{e}}_{loc},\\&\psi_i(z)\mapsto\psi_i^{-1}(z),\\&f(x_{[1,\gamma]})\mapsto(-1)^{|\gamma|}\psi_{[1,\gamma]}^{-1}(x_{[1,\gamma]})f(x_{[1,\gamma]}).
\end{array}\nonumber
\end{equation}

Let $\mathbf{H}^{\mathfrak{e},coop}_{loc}$ denote the algebra equipped with oppsite comultiplication defined as follows: let $\phi_i(z)=(-1)^{r_i+1}\psi_i(z)$, where $r_i$ is the number of loops at $i\in I$, and the $H^0$-action on $\mathbf{H}$ is given by $\phi_i(z)g\phi_i^{-1}(z)=g\frac{fac(z|x_{[1,\gamma]})}{fac(x_{[1,\gamma]}|z)}$ for $g\in\mathbf{H}_\gamma$. The comultiplication is defined as
\begin{equation}
\begin{array}{ll}
\Delta^{op}:&\mathbf{H}^{\mathfrak{e},coop}_{loc}\rightarrow\mathbf{H}^{\mathfrak{e},coop}_{loc}\otimes\mathbf{H}^{\mathfrak{e},coop}_{loc},\\&\phi_i(z)\mapsto\phi_i(z)\otimes\phi_i(z),\\&g(x_{[1,\gamma]})\mapsto\sum_{\gamma_1+\gamma_2=\gamma}\frac{g(x_{[1,\gamma_1]}\otimes x_{[\gamma_1+1,\gamma]})\phi_{[\gamma_1+1,\gamma]}(x_{[1,\gamma_1]})}{fac(x_{[\gamma_1+1,\gamma]}|x_{[1,\gamma_1]})}.
\end{array}\nonumber
\end{equation}
The antipode is given by

\begin{equation}
\begin{array}{ll}
S^{op}:&\mathbf{H}^{\mathfrak{e},coop}_{loc}\rightarrow\mathbf{H}^{\mathfrak{e},coop}_{loc},\\&\phi_i(Z)\mapsto\phi_i^{-1}(z),\\&g(x_{[1,\gamma]})\mapsto(-1)^{|\gamma|}g(x_{[1,\gamma]})\phi_{[1,\gamma]}^{-1}(x_{[1,\gamma]}).
\end{array}\nonumber
\end{equation}

One can verity that $\mathbf{H}^{\mathfrak{e}}_{loc}$ and $\mathbf{H}^{\mathfrak{e},coop}_{loc}$ are Hopf algebras (see \cite[Sec.~2.1]{RSYZ}). Now construct a bilinear skew-Hopf pairing $(\cdot,\cdot):\mathbf{H}^{\mathfrak{e}}_{loc}\otimes\mathbf{H}^{\mathfrak{e},coop}_{loc}\rightarrow\mathbb C$ in the following way:
\begin{itemize}
\item[$\bullet$]$(f_{\gamma_1},g_{\gamma_2})=0$ if $\gamma_1\neq\gamma_2$, and $(f_\gamma,\phi_i(z))=0$, $(\psi_i(z),g_\gamma)=0$,
\item[$\bullet$]$(\psi_k(u),\phi_l(w))=\frac{fac(u|w)}{fac(w|u)}$, for any $u,w\in I$,
\item[$\bullet$]$(f_{e_i},g_{e_i})=Res_{x=\infty}f(x_i)g(-x_i)dx$, where $dx$ is a top form,
\end{itemize}
and extend via $(a,bb')=(\Delta(a),b\otimes b')$ and $(aa',b)=(a\otimes a',\Delta^{op}(b))$, which leads to $$(f,g)=Res_{x=\infty}\frac{f(x_A)g(-x_A)}{|A|!fac(x_A)}dx.$$

It is known that if $A$, $B$ are Hopf algebras endowed with a skew-Hopf paring $(\cdot,\cdot)$, then there is a unique Hopf structure on $A\otimes B$, called the Drinfeld double of $(A,B,(\cdot,\cdot))$, determined by the following properties (\cite{Jos}):

\begin{itemize}
\item[$\bullet$]$(a\otimes1)(a'\otimes1)=aa'\otimes1$,
\item[$\bullet$]$(1\otimes b)(1\otimes b')=1\otimes bb'$,
\item[$\bullet$]$(a\otimes1)(1\otimes b)=a\otimes b$,
\item[$\bullet$]$(1\otimes b)(a\otimes1)=\sum(a_1,S_B(b_1))a_2\otimes b_2(a_3,b_3)$,
\end{itemize}
where $a,a'\in A$, $b,b'\in B$, and $\Delta_A^2(a)=\sum a_1\otimes a_2\otimes a_3$ and $\Delta_B^2(b)=\sum b_1\otimes b_2\otimes b_3$.

\begin{defn}The Drinfeld double of the shuffle algebra $D(\mathbf{H})$ is the Drinfeld double of $(\mathbf{H}^{\mathfrak{e}}_{loc},\mathbf{H}^{\mathfrak{e},coop}_{loc},(\cdot,\cdot))$. The Drinfeld double of the spherical shuffle algebra $D(\mathbf{H}^s)$ is defined in a similar way. Then the Drinfeld double of the COHA $D(\Hc)$ is defined to be $D(\mathbf{H})$, and the Drinfeld double of the spherical COHA $D(\Hc^s)$ is defined as $D(\mathbf{H}^s)$.
\end{defn}

\subsection{Donaldson-Thomas invariants}

This section recalls the Donaldson-Thomas series and invariants defined using COHA, which is not a major part of this paper. The uninterested readers could skip this section safely. Given a central charge $Z$, any non-zero representation $E$ admits a canonical Harder-Narasimhan filtration, i.e., $0=E_0\subset E_1\subset\cdots\subset E_n=E$ such that all the quotients $F_k=E_k/E_{k-1}, i=1,\ldots,n$ are semistable and $\Arg(F_1)>\cdots>\Arg(F_n)$. Fix a sector $V\subset\mathbb H_+=\{z\in\Cp\;|\;\Im(z)>0\}$, denote by $\mathbf M_{V,\gamma}\subset\mathbf M_\gamma$ the set of all representations whose Harder-Narasimhan factors have dimensions in $Z^{-1}(V)$ and $\Hc_V=\oplus_\gamma\Hc_{V,\gamma}=\oplus_\gamma\coh_{c,\mathbf G_\gamma}^{\bullet,crit}(\mathbf M_{V,\gamma},W_\gamma)^\vee$. Let $\mathcal R_+$ be the associative algebra generated over $\widehat{K_0}$ by $e_\gamma,\gamma\in\Z_{\geqslant0}^I$ satisfying $e_{\gamma_1}e_{\gamma_2}=\mathbb L^{-\chi(\gamma_1,\gamma_2)}e_{\gamma_1+\gamma_2}, e_0=1$, where $\widehat{K_0}$ is the $K_0$ group of $\mathscr{D}(\textbf{MMHS})$ completed by pure motives with weights approaching to $+\infty$. Then the motivic DT-series is defined as $A_V:=\sum_{\gamma\in\Z_{\geqslant0}^I}[\Hc_{V,\gamma}]e_\gamma$.

\begin{defn}Let $\Lambda_\theta^Z:=\{\gamma\in\Nn^I\mid\gamma=0\:\text{or}\:\Arg(Z(\gamma))=\theta\}\subset\Nn^I$ for a fixed $\theta$. A central charge $Z$ is $\theta$-generic if $\gamma_1,\gamma_2\in\Lambda_\theta^Z$ implies $\langle\gamma_1,\gamma_2\rangle=0$. It's called generic if it's $\theta$-generic for all $\theta\in(0,\pi)$.
\end{defn}

Recall the King stability defined in \cite{K}:

\begin{defn}
Given $\kappa=(\kappa_i)_{i\in I}\in\mathbb R^I$, a representation E of Q is $\kappa$-semistable (resp., stable) if $\kappa(E)=0$, and every proper subobject $F\subset E$ satisfies $\kappa(F)\leqslant0$ (resp., $\kappa(F)<0$).
\end{defn}

This is equivalent to the following stability. If $\kappa\in\QQ^I$ we can fix $m\in\Nn^I$ such that $m\kappa\in\Zn^I$. We linearize the $\mathbf G_\gamma$-action on $\mathbf M_\gamma$ via the trivial line bundle on $\mathbf M_\gamma$ and the character $\chi_\gamma:\mathbf G_\gamma\rightarrow\Cp^*,\;(g_i)_{i\in I}\mapsto\prod_{i\in I}det(g_i)^{-m\kappa_i}$, and let $\mathbf M_\gamma^{\kappa\text{-}ss}$ be the variety of semistable points w.r.t. this linearization. Then the GIT quotient $\mathcal M_\gamma^{\kappa\text{-}ss}=\mathbf M_\gamma^{\kappa\text{-}ss}\sslash_{\chi_\gamma}\mathbf G_\gamma$ is a coarse moduli space. Similarly $\mathcal M_\gamma^{\kappa\text{-}ss,sp}=\mathbf M_\gamma^{\kappa\text{-}ss,sp}\sslash_{\chi_\gamma}\mathbf G_\gamma$, where $\mathbf M_\gamma^{\kappa\text{-}ss,sp}=\mathbf M_\gamma^{\kappa\text{-}ss}\cap\mathbf M_\gamma^{sp}$ .

By \cite[Lem.~4.21]{DMSS}, for a fixed $\theta$-generic central charge $Z$ and a dimension vector $\gamma\in\Lambda_\theta^Z$, there is a King stability parameter $\kappa_Z\in\QQ^I$ such that a $Q$-representation is $Z$-(semi)stable if and only if it's $\kappa_Z$-(semi)stable. Thus there is a map $p_\gamma^{Z,sp}: \mathbf M_\gamma^{Z\text{-}ss,sp}/\mathbf G_\gamma\rightarrow\mathcal M_\gamma^{Z\text{-}ss,sp}:=\mathcal M_\gamma^{\kappa_Z\text{-}ss,sp}$ taking a $Q$-representation to its semisimplification in the category of $Z$-semistable $Q$-representations, i.e., the associated polystable object.

\begin{defn}{\rm(\cite{DM})}
For a $\theta$-generic central charge $Z$ and $\gamma\in\Lambda_\theta^Z\setminus\{0\}$ define the following elements in $\mathbf{MMHM}(\mathcal M_\gamma^{Z\text{-}ss})$ and $\mathscr D^b(\mathbf{MMHS})$ respectively:
$$\mathcal{BPS}_{W,\gamma}^Z=
\begin{cases}
\varphi_W^{mon}\mathcal{IC}_{\mathcal M_\gamma^{Z\text{-}ss}}(\QQ),& \mathcal M_\gamma^{Z\text{-}st}\neq\emptyset,\\
0,& otherwise.
\end{cases}$$
$$BPS_{W,\gamma}^{Z,sp}=\Ho_c(\mathcal M_\gamma^{Z\text{-}ss,sp},\mathcal{BPS}_{W,\gamma}^Z)^\vee.$$
The refined Donaldson-Thomas invariants is defined as $$\Omega_{W,\gamma}^{Z,sp}(q^{-1/2})=\chi_q(BPS_{W,\gamma}^{Z,sp},q^{1/2}),$$ where $\chi_q$ is the weight polynomial $\chi_q(\mathcal L,q^{1/2})=\sum_{i,j\in\Zn}\dim(Gr_j^W(\Ho^i(\mathcal L)))q^{j/2}$ for $\mathcal L\in\mathscr D^b(\mathbf{MMHS})$.
\end{defn}


\subsection{Scattering diagram}

The scattering diagram (a.k.a. wall-crossing structure) was introduced in \cite{KoSo} and studied by many authors, e.g., \cite{Bou, B1, DMa, GHKK, M}. We follow \cite{B1} to briefly review this notion.

Let $N\cong\Nn^n$ be a lattice with a fixed basis $(e_i)_{i=1}^n$, and set $N^\oplus=\{\sum_{i=1}^na_ie_i\mid a_i\geqslant0\}$ and $N^+=N^\oplus\setminus\{0\}$. For the dual $M=\Hom_\Zn(N,\Zn)$ let $M_\mathbb{R}=M\otimes\mathbb{R}$, $M_\mathbb{R}^+=\{y\in M_\mathbb{R}\mid y(\gamma)>0,\forall \gamma\in N^+\}$, $M^+=M_\mathbb{R}^+\cap M$,$M^\oplus=M^+\cup\{0\}$ and $M_\mathbb{R}^-=-M_\mathbb{R}^+$. We fix an $N^+$-graded nilpotent Lie algebra $\mathfrak g=\oplus_{\gamma\in N^+}\mathfrak g_\gamma$, Then there is a corresponding algebraic group $G$ with a bijection $exp:\mathfrak g\rightarrow G$.

A cone in $M_\mathbb{R}$ is a subset of the form $\sigma=\{\sum_{k=1}^pa_ky_k\mid a_k\in\mathbb R_\geqslant0\}=\{y\in M_\mathbb{R}\mid y(\gamma_k)\geqslant0, k=1,\ldots,q\}$ for some $y_k\in M$ and $\gamma_k\in N$. The codinemsion of a cone is the codimension of the subspace of $M_\mathbb{R}$ it spans, and the codimension one cones are called walls and are denoted by $\mathfrak d$. A face of a cone $\sigma$ is a subset $\sigma\cap\gamma^\perp=\{y\in\sigma\mid y(\gamma)=0\}$ for a $\gamma\in N$ satsfying $y(\gamma)\geqslant0,\forall y\in\sigma$.

\begin{defn}
A cone complex in $M_\mathbb{R}$ is a finite collection $\mathfrak S$ of cones such that any face of a cone in $\mathfrak S$ is also a cone in $\mathfrak S$, and the intersection of any two cones in $\mathfrak S$ is a face of each. The support of a cone complex is the closed subset
$supp(\mathfrak S)=\cup_{\sigma\in\mathfrak{S}}\sigma\subset M_\mathbb{R}$.
\end{defn}

For a cone $\sigma$ let $\mathfrak g(\sigma)=\oplus_{\gamma\in N^+\cap\sigma^\perp}\mathfrak g_\gamma\subset\mathfrak g$, where $\sigma^\perp=\{\gamma\in N\mid y(\gamma)=0,\forall y\in\sigma\}$

\begin{defn}
A $\mathfrak g$-complex $\mathfrak D=(\mathfrak S,f)$ is a cone complex $\mathfrak S$ in $M_\mathbb{R}$ with a choice
$f_\mathfrak{d}\in\mathfrak g(\mathfrak{d})$ for each wall $\mathfrak d\in\mathfrak S$.

The essential support of $\mathfrak D=(\mathfrak S, f)$ is the subset $supp_{ess}(\mathfrak D)=\cup_{f_\mathfrak{d}\neq0}\mathfrak d\subset supp(\mathfrak S)$
\end{defn}

A smooth path $p:[0,1]\rightarrow M_\mathbb{R}$ is called $\mathfrak D$-generic if the endpoints $p(0)$ and $p(1)$ do not lie in $supp_{ess}(\mathfrak D)$, $p$ does not meet any cones of $\mathfrak S$ of codimension greater than one, and all intersections of $p$ with $\mathfrak d$ are transversal. Thus there exist $0<t_1<\cdots<t_k <1$ such that $p(t_i)\in supp_{ess}(\mathfrak D)$, and at each $t_i$ there is a unique wall $\mathfrak d_i\in\mathfrak S$ such that $p(t_i)\in\mathfrak d_i$. We denote $F_\mathfrak{D}(\mathfrak d)=exp(f_\mathfrak{d})$ and $F_\mathfrak{D}(p)=F_\mathfrak{D}(\mathfrak d_k)^{\epsilon_k}\cdots F_\mathfrak{D}(\mathfrak d_1)^{\epsilon_1}\in G$, where $\epsilon_i\in\{\pm1\}$ is the negative of the sign of the derivative of $p(t)(\gamma)$ at $t=t_i$.

\begin{defn}
A $\mathfrak g$-complex $\mathfrak D$ is called consistent if for any two $\mathfrak D$-generic paths $p_i$ with the same endpoints, we have $F_\mathfrak{D}(p_1)=F_\mathfrak{D}(p_2)$.

Two $\mathfrak g$-complex $\mathfrak D_1$ and $\mathfrak D_2$ are equivalent if $F_{\mathfrak D_1}(p)=F_{\mathfrak D_2}(p)$ for any path $p:[0,1]\rightarrow M_\mathbb{R}$ which is both $\mathfrak D_1$-generic and $\mathfrak D_2$-generic.
\end{defn}

Next let's remove the nilpotency assumption on $\mathfrak g$, and denote the ideal $\mathfrak g^{>k}=\oplus_{|\gamma|>k}\mathfrak g_\gamma$, where $|\gamma|=\sum_i\gamma^i$. Then $\mathfrak g_{\leqslant k}=\mathfrak g/\mathfrak g^{>k}$ is nilpotent, and the corresponding algebraic group $G_{\leqslant k}$ is unipotent. Denoting $\pi^{lk}:\mathfrak g_{\leqslant l}\rightarrow\mathfrak g_{\leqslant k}$, $\pi^{lk}:G_{\leqslant l}\rightarrow G_{\leqslant k}$, and \[\hat{\mathfrak g}=\lim_\leftarrow\mathfrak g_{\leqslant k},\;\hat G=\lim_\leftarrow G_{\leqslant k}\] the pro-nilpotent Lie algebra and the corresponding pro-unipotent algebraic group, we obtain a bijection $exp:\hat{\mathfrak g}\rightarrow\hat G$.

\begin{defn}
A scattering diagram (also called $\hat{\mathfrak g}$-complex) $\mathfrak D$ is a sequence of $\mathfrak g_{\leqslant k}$-complexes $\mathfrak D_k$ for $k\geqslant1$ such that for any $k<l$, the $\mathfrak g_{\geqslant k}$-complexes $\pi_*^{lk}(\mathfrak D_l):=(\mathfrak S,\pi^{lk}\circ f)$ and $\mathfrak D_k$ are equivalent. We say that $\mathfrak D$ is consistent if each $\mathfrak D_k$ is consistent. We say that $\mathfrak D$ is equivalent to another $\hat{\mathfrak g}$-complex $\mathfrak D'=(\mathfrak D_k')_{k\geqslant1}$ if $\mathfrak D_k$ and $\mathfrak D_k'$ are equivalent for all $k\geqslant1$.
\end{defn}

We denote $supp(\mathfrak D)=\cup_{k\geqslant1}supp(\mathfrak D_k)$ and $supp_{ess}(\mathfrak D)=\cup_{k\geqslant1}supp_{ess}(\mathfrak D_k)$.

We have the following

\begin{prop}{\rm(\cite{KoSo})}
There is a bijection between equivalence classes of consistent $\hat{\mathfrak g}$-complexes and elements of the group $\hat G$.
\end{prop}

Now let $\mathfrak g_{COHA}=\Hc_{>0}=\oplus_{|\gamma|>0}\Hc_\gamma$ be the Lie algebra in the COHA $\Hc$ under the commutator bracket, then the completion $\hat{\mathfrak g}_{COHA}$ is a pro-nilpotent Lie lagebra and the corresponding pro-unipotent group is denoted by $\hat G_{COHA}$. Then we have the following COHA scattering diagram

\begin{prop}{\rm(\cite{B1})}
There is a consistent $\hat{\mathfrak g}_{COHA}$ scattering diagram $\mathfrak D_{COHA}$ such that
\begin{itemize}
\item[a)] the support consists of $\kappa\in M_{\mathbb R}$ for which there exist non-zero $\kappa$-semistable $Q$-representations,
\item[b)] the wall-crossing automorphism at a general point $\kappa\in\mathfrak{d}\subset supp(\mathfrak D_{COHA})$ is $F_{\mathfrak D}(\mathfrak d)=\Hc^{Z_\kappa\text{-}ss}_{re^{i\theta}}=\oplus_{\gamma\in\Lambda_\theta^{Z_\kappa}}\coh_{c,\mathbf{G}_\gamma}^{\bullet,crit}(\mathbf{M}_\gamma^{Z_\kappa\text{-}ss},W_\gamma)^\vee\otimes\T^{\dim\mathbf{M}_\gamma/\mathbf{G}_\gamma}$, where $Z_\kappa=-\kappa+\sqrt{-1}|\cdot|$.
\end{itemize}
\end{prop}

\begin{rem}
Indeed Bridgeland used the motivic Hall algebra in \cite{B1} instead of COHA we use here, but the proof is similar.
\end{rem}

\subsection{Dimensional reduction}\label{DimRed}

We briefly remind the dimensional reduction from 3CY to 2CY, and that the multiplication of the 3-dimensional COHA, i.e., the COHA of a quiver with potential, induces the multiplication of the 2-dimensional COHA, i.e., the COHA of the preprojective algebra. See \cite{KoSo3,D1,RS} for details.

Fix a quiver $Q$ with the set of vertices $I$ and the set of arrows $\Omega$, one constructs the double quiver $\overline{Q}$, the preprojective algebra $\Pi_{Q}$, and the triple quiver with potential $(\widetilde{Q}, W)$ as follows. For any arrow $a:i\rightarrow j\in\Omega$, we add an inverse arrow $a^{*}:j\rightarrow i$ to $Q$ to get $\overline{Q}$, then $\Pi_{Q}=\mathbb{C}\overline{Q}/\sum_{a\in\Omega}[a, a^{*}]$. Adding loops $l_{i}: i\rightarrow i$ at each vertex $i\in I$ to $\overline{Q}$ gives us the ``triple" quiver $\widetilde{Q}$. It is endowed with the cubic potential $W=\sum_{a\in\Omega}[a, a^{*}]l$, where $l=\sum_{i\in I}l_{i}$. For any dimension vector
$\gamma=(\gamma^{i})_{i\in I}\in \mathbb{Z}^{I}_{\geqslant0}$ we have the following algebraic varieties:
\begin{itemize}
\item[a)] the space $\textbf{M}_{\overline{Q},\gamma}$ of representations of the double quiver $\overline{Q}$  in coordinate spaces $({\mathbb C}^{\gamma^i})_{i\in I}$ ;
\item[b)] the similar space of representations $\textbf{M}_{\Pi_{Q},\gamma}$ of $\Pi_{Q}$;
\item[c)] the similar space of representations $\textbf{M}_{\widetilde{Q},\gamma}$ of $\widetilde{Q}$.
\end{itemize}
All these spaces of representations are endowed with the
action by conjugation of the complex algebraic group
$\textbf{G}_{\gamma}=\prod_{i\in I}GL(\gamma^{i}, \mathbb{C})$.

We have $\textbf{M}_{\widetilde{Q},\gamma}=\textbf{M}_{\overline{Q},\gamma}\times\mathbb{A}^{\gamma\cdot\gamma}$ (dot denotes the inner product), and $f=Tr(W)_{\gamma}=\sum\limits_{i\in
I,k=1,\ldots,(\gamma^{i})^{2}}f_{ik}x_{ik}$, where $f_{ik}$ are functions on $\textbf{M}_{\overline{Q},\gamma}$, and $\{x_{ik}\}$ is a linear coordinate system on $\mathbb{A}^{\gamma\cdot\gamma}$. Then
$\textbf{M}_{\Pi_{Q},\gamma}$ is the reduced scheme which is the vanishing locus of all functions $f_{k}$. Denote by
$\textbf{M}_{\Pi_{Q},\gamma_{1},\gamma_{2}}$ the space of
representations of $\overline{Q}$ in coordinate spaces of dimension $\gamma_{1}+\gamma_{2}$ such that the standard coordinate subspaces of dimension $\gamma_{1}$ form a subrepresentation, and the restriction of $\rho\in\textbf{M}_{\Pi_{Q},\gamma_{1},\gamma_{2}}$
on the block-diagonal part is an element in
$\textbf{M}_{\Pi_{Q},\gamma_1}\times\textbf{M}_{\Pi_{Q},\gamma_2}$.
The group
$\textbf{G}_{\gamma_{1},\gamma_{2}}\subset\textbf{G}_{\gamma}$
consisting of transformations preserving subspaces
$(\mathbb{C}^{\gamma_{1}^{i}}\subset\mathbb{C}^{\gamma^{i}})_{i\in I}$ acts on $\textbf{M}_{\Pi_{Q},\gamma_{1},\gamma_{2}}$. Suppose that we are given a collection of $\textbf{G}_{\gamma}$-invariant closed subsets
$\textbf{M}_{\overline{Q},\gamma}^{sp}\subset\textbf{M}_{\overline{Q},\gamma}$ satisfying the following condition: for any short exact sequence $0\rightarrow E_{1}\rightarrow E\rightarrow E_{2}\rightarrow0$ of representations of $\overline{Q}$ with dimension vectors
$\gamma_{1}, \gamma:=\gamma_{1}+\gamma_{2}, \gamma_{2}$
respectively, $E\in\textbf{M}_{\overline{Q},\gamma}^{sp}$ if and only if $E_{1}\in\textbf{M}_{\overline{Q},\gamma_{1}}^{sp}$, and
$E_{2}\in\textbf{M}_{\overline{Q},\gamma_{2}}^{sp}$. Then
$\textbf{M}_{\widetilde{Q},\gamma}^{sp}=\textbf{M}_{\overline{Q},\gamma}^{sp}\times\mathbb{A}^{\gamma\cdot\gamma}$.

Then there is an isomorphism

$$\coh_{c,G}^{\bullet,crit}(\textbf{M}_{\widetilde{Q},\gamma}^{sp},W_\gamma)\simeq\coh_{c,G}^{\bullet}(\textbf{M}_{\Pi_{Q},\gamma}^{sp}\times\mathbb{A}^{n}, \mathbb{Q}).$$

\begin{thm}[\cite{RS}]
The COHA of $(\widetilde{Q},W)$ induces the associative
multiplication on $\mathcal{H}_{\Pi_Q}=\bigoplus\limits_{\gamma\in\mathbb{Z}^{I}_{\geqslant0}}\mathcal{H}_{\Pi_Q,\gamma}$, where $\mathcal{H}_{\Pi_Q,\gamma}:=\coh_{c,\textbf{G}_{\gamma}}^{\bullet}(\textbf{M}_{\Pi_{Q},\gamma}^{sp},\mathbb{Q})^{\vee}\otimes\mathbb{T}^{-\chi_{Q}(\gamma,\gamma)}$.

This multiplication preserves the modified cohomological degree, thus the zero degree part
$\mathcal{H}^0_{\Pi_Q}=\bigoplus\limits_{\gamma\in\mathbb{Z}^{I}_{\geqslant0}}\mathcal{H}_{\Pi_Q,\gamma}^0=\bigoplus\limits_{\gamma\in\mathbb{Z}^{I}_{\geqslant0}}\coh_{c,\textbf{G}_{\gamma}}^{-2\chi_{Q}(\gamma,\gamma)}(\textbf{M}_{\Pi_{Q},\gamma}^{sp},\mathbb{Q})^{\vee}\otimes\mathbb{T}^{-\chi_{Q}(\gamma,\gamma)}$
is a subalgebra of $\mathcal{H}_{\Pi_Q}$.
\end{thm}

We call $\mathcal{H}_{\Pi_Q}$ the {\it
Cohomological Hall algebra}  of the preprojective algebra $\Pi_{Q}$ associated with the quiver $Q$.

\section{Edge contraction}

We establish an algebra morphism among the COHA's induced by an edge contraction. We further show that this morphism is compatible with the bialgebra structures, the Drinfeld doubles, the DT theory and the dimensional reductions.

\subsection{Edge contraction of a quiver with potential}

Given a quiver with potential $(Q,W)$ and fix an arrow $a_0:i_+\rightarrow i_-$ for a distinct pair of vertices $i_+$ and $i_-$, we define a new quiver with potential $(\widehat{Q},\widehat{W})$, i.e., the edge contraction of $(Q,W)$, as follows. The set of vertices $\hat{I}=I\setminus\{i_-\}$, and the set of arrows 
\begin{equation}
\begin{array}{ll}
\widehat{\Omega}=&\Omega\cup\{\hat a:=aa_0:s(a_0)\rightarrow t(a)|s(a)=t(a_0)\}\\&\cup\{\hat a:=a_0^{-1}a:s(a)\rightarrow s(a_0)|t(a)=t(a_0)\}\\&\cup\{\hat a:=a_0^{-1}aa_0:s(a_0)\rightarrow s(a_0)|s(a)=t(a)=t(a_0)\}\\&\setminus\{a|s(a)=t(a_0)\:or\:t(a)=t(a_0)\},
\end{array}\nonumber
\end{equation}
where $a_0^{-1}$ is a formal symbol which gives $M_{a_0}^{-1}$ for a fixed $Q$-representation $(M_a)_{a\in\Omega}\in\mathbf{M}_\gamma^\heartsuit$, so $a_0^{-1}$ can be viewed as an arrow $t(a_0)\rightarrow s(a_0)$ obtained by reverseing $a_0$.
To get the potential $\widehat{W}$, one just replaces the arrows $a$ appear in $W$ by the corresponding $\hat a$.

Note that we sometimes denote the vertex $i_+$ in $\widehat{Q}$ by $i_0$ in order to avoid confusion. We will use $i_+$ and $i_0$ interchangeably.

Given a representation $M=(M_a)_{a\in\Omega}$ of $Q$ in $\mathbf{M}_\gamma^\heartsuit$, define a representation $\widehat{M}=(\widehat{M}_a)_{a\in\widehat{\Omega}}$ of $\widehat{Q}$ in the following way: $\widehat{M}_a=M_a$ if $a\in\widehat{\Omega}\cap\Omega$, $\widehat{M}_{aa_0}=M_aM_{a_0}$ if $aa_0\in\widehat{\Omega}$ and $s(a)=t(a_0)$, $\widehat{M}_{a_0^{-1}a}=M_{a_0}^{-1}M_a$ if $a_0^{-1}a\in\widehat{\Omega}$ and $t(a)=t(a_0)$, and $\widehat{M}_{a_0^{-1}aa_0}=M_{a_0}^{-1}M_aM_{a_0}$ if $s(a)=t(a)=t(a_0)$.

\begin{example}
Consider the following quiver
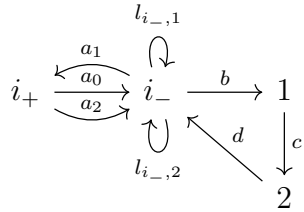
\[
\begin{tikzcd}
i_+\arrow[r,"a_0"]\arrow[r,bend right,"a_2"] & i_-\arrow[r,"b"]\arrow[l,bend right,swap,"a_1"]\arrow[l,loop above,"l_{i_-,1}"]\arrow[l,loop below,"l_{i_-,2}"] & 1\arrow[d,"c"]\\&&2\arrow[lu,swap,"d"]
\end{tikzcd}
\]
with potential $W=a_1l_{i_-,1}^2l_{i_-,2}^3a_0+l_{i_-,1}dcb$. After contracting along $a_0$ we get the quiver
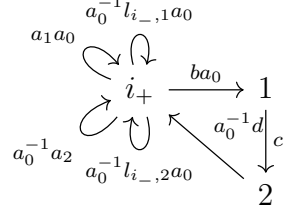
\[
\begin{tikzcd}
i_+\arrow[out=230,in=200,loop,"a_0^{-1}a_2"]\arrow[r,"ba_0"]\arrow[out=160,in=130,loop,"a_1a_0"]\arrow[l,loop above,"a_0^{-1}l_{i_-,1}a_0"]\arrow[l,loop below,"a_0^{-1}l_{i_-,2}a_0"] & 1\arrow[d,"c"]\\&2\arrow[lu,swap,"a_0^{-1}d"]
\end{tikzcd}
\]
and the potential $\widehat{W}=(a_1a_0)(a_0^{-1}l_{i_-,1}a_0)^2(a_0^{-1}l_{i_-,2}a_0)^3+(a_0^{-1}l_{i_-,1}a_0)(a_0^{-1}d)c(ba_0)$.
\end{example}

\begin{rem}
In quiver gauge theory, edge contraction is the process of Higgsing by giving the bifundamental field $a_0$ a vacuum expectation value (vev). If $a_0$ only appears in terms with four or more arrows in $W$, then Higgsing is exactly the same as edge contraction. If $a_0$ appears in a cubic term $a_0bc$ in $W$, then this vev has the effect of giving a mass to $b$ and $c$. We take cyclic partial derivatives $\frac{\partial W}{\partial b}=c+W_1$ and $\frac{\partial W}{\partial c}=b+W_2$, then substitute $c=-W_1,\;b=-W_2$ back to $W$ to obtain $\widehat W$. Since the sheaf of vanishing cycles is supported on the critical locus of $W$, this massive case is indeed equivalent to the edge contraction.
\end{rem}

\subsection{Compatibility with multiplication and comultiplication}

Fix a dimension vector $\gamma=(\gamma^i)_{i\in I}$ of $Q$-representations with $\gamma^{i_+}=\gamma^{i_-}$, and a framing vector $\omega=(\omega^i)_{i\in I}$. Let $\hat{\gamma}=(\hat{\gamma}^i)_{i\in I\setminus\{i_-\}}=(\gamma^i)_{i\in I\setminus\{i_-\}}$, and $\hat{\omega}=(\hat{\omega}^i)_{i\in I\setminus\{i_-\}}=(\omega^i, i\in I\setminus\{i_+,i_-\};\;\hat{\omega}^{i_+}=\omega^{i_+}+\omega^{i_-})$, which are dimension vector and framing vector of $\widehat{Q}$-representations. Then $\mathcal{M}_{\gamma, \omega}^{\circ,\heartsuit}=\mathbf{M}_{\gamma, \omega}^{\circ,\heartsuit}/\mathbf{G}_\gamma\cong(\mathbf{M}_{\gamma,\omega}^{\circ,\heartsuit}/\GL_{\gamma_{i_-}}(\mathbb{C}))/\mathbf{G}_{\hat{\gamma}}\cong\mathbf{M}_{\hat{\gamma},\hat{\omega}}^{\circ}/\mathbf{G}_{\hat{\gamma}}=\mathcal{M}_{\hat{\gamma},\hat{\omega}}^{\circ}$, and the isomorphism is denoted by $\epsilon_{\gamma,\omega}: \mathcal{M}_{\gamma,\omega}^{\circ,\heartsuit}\rightarrow\mathcal{M}_{\hat{\gamma},\hat{\omega}}^\circ$.

Let's first check the compatibility of the inclusions $\iota$ (with various subscripts in the following diagrams corresponding to the dimension vectors) induced by $\mathbf{M}_{\gamma,\omega}^{\circ,\heartsuit}\hookrightarrow\mathbf{M}_{\gamma,\omega}^\circ$ with the multiplication of the COHA. Note that $\mathbf{M}_{\gamma_1,\gamma_2}^\heartsuit:=\mathbf{M}_\gamma^\heartsuit\cap\mathbf{M}_{\gamma_1,\gamma_2}=\widetilde{\mathbf{M}}_{\gamma_1,\gamma_2}^\heartsuit:=p^{-1}(\mathbf{M}_{\gamma_1}^\heartsuit\times\mathbf{M}_{\gamma_2}^\heartsuit)$, where $p:\mathbf{M}_{\gamma_1,\gamma_2}\rightarrow\mathbf{M}_{\gamma_1}\times\mathbf{M}_{\gamma_2}$ is the natural projection.

\begin{itemize}
\item[$\bullet$]For $\alpha=p_{\gamma_{1},\gamma_{2},\omega_1,\omega_2}^{\circ,*}$, the following diagram commutes 
$$\begin{xy}
(0,20)*+{\mathbf{M}_{\gamma_{1},\gamma_{2}; \omega_1, \omega_2}^{\circ}/\mathbf{G}_{\gamma_1}\times\mathbf{G}_{\gamma_2}}="v1";
(80,20)*+{(\mathbf{M}_{\gamma_1, \omega_1}^{\circ}\times\mathbf{M}_{\gamma_2,\omega_2}^{\circ})/\mathbf{G}_{\gamma_1}\times\mathbf{G}_{\gamma_2}}="v2";
(0,0)*+{\mathbf{M}_{\gamma_{1},\gamma_{2}; \omega_1, \omega_2}^{\circ,\heartsuit}/\mathbf{G}_{\gamma_1}\times\mathbf{G}_{\gamma_2}}="v3";
(80,0)*+{(\mathbf{M}_{\gamma_1, \omega_1}^{\circ,\heartsuit}\times\mathbf{M}_{\gamma_2, \omega_2}^{\circ,\heartsuit})/\mathbf{G}_{\gamma_2}\times\mathbf{G}_{\gamma_2}}="v4";
{\ar@{->}^{p_{\gamma_{1},\gamma_{2},\omega_1,\omega_2}^\circ} "v1";"v2"};
{\ar@{->}^{p_{\gamma_{1},\gamma_{2},\omega_1,\omega_2}^{\circ,\heartsuit}}
"v3";"v4"};{\ar@{->}_{\iota_{1,2;1\times 2}} "v3";"v1"}; {\ar@{->}_{\iota_{1\times 2}} "v4";"v2"}
\end{xy}$$ so does the diagram $$\begin{xy}
(0,20)*+{\coh_{c,\mathbf{G}_{\gamma_1}\times\mathbf{G}_{\gamma_2}}^{\bullet,crit}(\mathbf{M}_{\gamma_{1},\gamma_{2}}^{sp},W_{\gamma_1,\gamma_2})^\vee\otimes\T^{d_1}}="v1";
(80,20)*+{\coh_{c,\mathbf{G}_{\gamma_1}\times\mathbf{G}_{\gamma_2}}^{\bullet,crit}(\mathbf{M}_{\gamma_1}^{sp}\times\mathbf{M}_{\gamma_2}^{sp},W_{\gamma_1}\boxplus W_{\gamma_2})^\vee}="v2";
(0,0)*+{\coh_{c,\mathbf{G}_{\gamma_1}\times\mathbf{G}_{\gamma_2}}^{\bullet,crit}(\mathbf{M}_{\gamma_{1},\gamma_{2}}^{sp,\heartsuit},W_{\gamma_1,\gamma_2})^\vee\otimes\T^{d_1}}="v3";
(80,0)*+{\coh_{c,\mathbf{G}_{\gamma_1}\times\mathbf{G}_{\gamma_2}}^{\bullet,crit}(\mathbf{M}_{\gamma_1}^{sp,\heartsuit}\times\mathbf{M}_{\gamma_2}^{sp,\heartsuit},W_{\gamma_1}\boxplus W_{\gamma_2})^\vee}="v4";
{\ar@{->}_{\alpha} "v2";"v1"};
{\ar@{->}_{\alpha^\heartsuit}
"v4";"v3"};{\ar@{->}^{\iota_{1,2;1\times 2}^*} "v1";"v3"}; {\ar@{->}^{\iota_{1\times 2}^*} "v2";"v4"}
\end{xy}$$
where $d_1=\sum_{a\in\Omega}\gamma_2^{s(a)}\gamma_1^{t(a)}$.

\item[$\bullet$]For $\beta=q_{\gamma_{1},\gamma_{2},\omega_1,\omega_2}^{\circ,*}$, similarly as above, the commutativity of $$\begin{xy}
(0,20)*+{\mathbf{M}_{\gamma_{1},\gamma_{2}; \omega_1, \omega_2}^{\circ}/\mathbf{G}_{\gamma_1}\times\mathbf{G}_{\gamma_2}}="v1";
(80,20)*+{\mathbf{M}_{\gamma_1, \gamma_2; \omega_1, \omega_2}^{\circ}/\mathbf{G}_{\gamma_1,\gamma_2}}="v2";
(0,0)*+{\mathbf{M}_{\gamma_{1},\gamma_{2}; \omega_1, \omega_2}^{\circ,\heartsuit}/\mathbf{G}_{\gamma_1}\times\mathbf{G}_{\gamma_2}}="v3";
(80,0)*+{\mathbf{M}_{\gamma_1, \gamma_2; \omega_1, \omega_2}^{\circ,\heartsuit}/\mathbf{G}_{\gamma_1,\gamma_2}}="v4";
{\ar@{->}^{q_{\gamma_{1},\gamma_{2},\omega_1,\omega_2}^\circ} "v1";"v2"};
{\ar@{->}^{q_{\gamma_{1},\gamma_{2},\omega_1,\omega_2}^{\circ,\heartsuit}}
"v3";"v4"};{\ar@{->}_{\iota_{1,2;1\times 2}}
"v3";"v1"}; {\ar@{->}_{\iota_{1,2}} "v4";"v2"}
\end{xy}$$
implies the commutativity of $$\begin{xy}
(0,20)*+{\coh_{c,\mathbf{G}_{\gamma_1}\times\mathbf{G}_{\gamma_2}}^{\bullet,crit}(\mathbf{M}_{\gamma_{1},\gamma_{2}}^{sp},W_{\gamma_1,\gamma_2})^\vee\otimes\T^{d_0}}="v1";
(80,20)*+{\coh_{c,\mathbf{G}_{\gamma_1,\gamma_2}}^{\bullet,crit}(\mathbf{M}_{\gamma_1, \gamma_2}^{sp},W_{\gamma_1,\gamma_2})^\vee}="v2";
(0,0)*+{\coh_{c,\mathbf{G}_{\gamma_1}\times\mathbf{G}_{\gamma_2}}^{\bullet,crit}(\mathbf{M}_{\gamma_{1},\gamma_{2}}^{sp,\heartsuit},W_{\gamma_1,\gamma_2})^\vee\otimes\T^{d_0}}="v3";
(80,0)*+{\coh_{c,\mathbf{G}_{\gamma_1,\gamma_2}}^{\bullet,crit}(\mathbf{M}_{\gamma_1, \gamma_2}^{sp,\heartsuit},W_{\gamma_1,\gamma_2})^\vee}="v4";
{\ar@{->}_{\beta} "v2";"v1"};
{\ar@{->}_{\beta^\heartsuit} "v4";"v3"};{\ar@{->}^{\iota_{1,2;1\times 2}^*}
"v1";"v3"}; {\ar@{->}^{\iota_{1,2}^*} "v2";"v4"}
\end{xy}$$
where $d_0=\sum_{i\in I}\gamma_2^i\gamma_1^i$.

\item[$\bullet$]For $\zeta=j_{\gamma_{1},\gamma_{2},\omega_1,\omega_2,*}^\circ$, since the diagram $$\begin{xy}
(0,20)*+{\mathbf{M}_{\gamma_1, \gamma_2; \omega_1, \omega_2}^{\circ}/\mathbf{G}_{\gamma_1,\gamma_2}}="v1";
(80,20)*+{\mathbf{M}_{\gamma,\omega}^{\circ}/\mathbf{G}_{\gamma_1,\gamma_2}}="v2";
(0,0)*+{\mathbf{M}_{\gamma_1, \gamma_2; \omega_1, \omega_2}^{\circ,\heartsuit}/\mathbf{G}_{\gamma_1,\gamma_2}}="v3";
(80,0)*+{\mathbf{M}_{\gamma,\omega}^{\circ,\heartsuit}/\mathbf{G}_{\gamma_1,\gamma_2}}="v4";
{\ar@{->}^{j_{\gamma_{1},\gamma_{2},\omega_1,\omega_2}^\circ} "v1";"v2"};
{\ar@{->}^{j_{\gamma_{1},\gamma_{2},\omega_1,\omega_2}^{\circ,\heartsuit}}
"v3";"v4"};{\ar@{->}_{\iota_{1,2}}
"v3";"v1"}; {\ar@{->}_{\iota_{\gamma;1,2}} "v4";"v2"}
\end{xy}$$ is Cartesian, the following diagram commutes by Theorem \ref{base}: $$\begin{xy}
(0,20)*+{\coh_{c,\mathbf{G}_{\gamma_1,\gamma_2}}^{\bullet,crit}(\mathbf{M}_{\gamma_1, \gamma_2}^{sp},W_{\gamma_1,\gamma_2})^\vee}="v1";
(80,20)*+{\coh_{c,\mathbf{G}_{\gamma_1,\gamma_2}}^{\bullet,crit}(\mathbf{M}_\gamma^{sp},W_\gamma)^\vee}="v2";
(0,0)*+{\coh_{c,\mathbf{G}_{\gamma_1,\gamma_2}}^{\bullet,crit}(\mathbf{M}_{\gamma_1, \gamma_2}^{sp,\heartsuit},W_{\gamma_1,\gamma_2})^\vee}="v3";
(80,0)*+{\coh_{c,\mathbf{G}_{\gamma_1,\gamma_2}}^{\bullet,crit}(\mathbf{M}_\gamma^{sp,\heartsuit},W_\gamma)^\vee}="v4";
{\ar@{->}^{\zeta} "v1";"v2"};
{\ar@{->}^{\zeta^\heartsuit}
"v3";"v4"};{\ar@{->}^{\iota_{1,2}^*}
"v1";"v3"}; {\ar@{->}^{\iota_{\gamma;1,2}^*} "v2";"v4"}
\end{xy}$$

\item[$\bullet$]For $\delta=pr_{\gamma_{1},\gamma_{2},\omega_1,\omega_2,*}^\circ$, similar as $\zeta$, the Cartesian diagram $$\begin{xy}
(0,20)*+{\mathbf{M}_{\gamma,\omega}^{\circ}/\mathbf{G}_{\gamma_1,\gamma_2}}="v1";
(80,20)*+{\mathbf{M}_{\gamma,\omega}^{\circ}/\mathbf{G}_\gamma}="v2";
(0,0)*+{\mathbf{M}_{\gamma,\omega}^{\circ,\heartsuit}/\mathbf{G}_{\gamma_1,\gamma_2}}="v3";
(80,0)*+{\mathbf{M}_{\gamma,\omega}^{\circ,\heartsuit}/\mathbf{G}_\gamma}="v4";
{\ar@{->}^{pr_{\gamma_{1},\gamma_{2},\omega_1,\omega_2}^\circ} "v1";"v2"};
{\ar@{->}^{pr_{\gamma_{1},\gamma_{2},\omega_1,\omega_2}^{\circ,\heartsuit}}
"v3";"v4"};{\ar@{->}_{\iota_{\gamma;1,2}}
"v3";"v1"}; {\ar@{->}_{\iota_\gamma} "v4";"v2"}
\end{xy}$$ implies the commutativity of the diagram $$\begin{xy}
(0,20)*+{\coh_{c,\mathbf{G}_{\gamma_1,\gamma_2}}^{\bullet,crit}(\mathbf{M}_\gamma^{sp},W_\gamma)^\vee}="v1";
(80,20)*+{\coh_{c,\mathbf{G}_\gamma}^{\bullet,crit}(\mathbf{M}_\gamma^{sp},W_\gamma)^\vee}="v2";
(0,0)*+{\coh_{c,\mathbf{G}_{\gamma_1,\gamma_2}}^{\bullet,crit}(\mathbf{M}_\gamma^{sp,\heartsuit},W_\gamma)^\vee}="v3";
(80,0)*+{\coh_{c,\mathbf{G}_\gamma}^{\bullet,crit}(\mathbf{M}_\gamma^{sp,\heartsuit},W_\gamma)^\vee}="v4";
{\ar@{->}^{\delta} "v1";"v2"};
{\ar@{->}^{\delta^\heartsuit} "v3";"v4"};{\ar@{->}^{\iota_{\gamma;1,2}^*}
"v1";"v3"}; {\ar@{->}^{\iota_\gamma^*} "v2";"v4"}
\end{xy}$$

\item[$\bullet$]Thom-Sebastiani isomorphism: pullback the sheaves in (\ref{TSsheaves}) from $\mathbf{M}_{\gamma,\omega}^\circ$ to $\mathbf{M}_{\gamma,\omega}^{\circ,\heartsuit}$, and by (\ref{PB}) we can see that the Thom-Sebastiani isomorphism is compatible with $\iota$.
\end{itemize}

Next let's check the compatibility of the maps $\epsilon$ (with various subscripts in the following diagrams corresponding to the dimension vectors) with the multiplication of the COHA. Note that $\varphi_{\tr W_\gamma}$ doesn't change under the edge contraction. Indeed, after replacing each arrow $a$ appearing in $W$ by $\hat a$, the trace of the corresponding linear map remains the same.

\begin{itemize}
\item[$\bullet$]For $\alpha$, the perimeter of the following diagram is Cartesian: $$\begin{xy}
(10,20)*+{\mathbf{M}_{\gamma_{1},\gamma_{2}; \omega_1, \omega_2}^{\circ,\heartsuit}/\mathbf{G}_{\gamma_1}\times\mathbf{G}_{\gamma_2}}="v1";
(80,20)*+{(\mathbf{M}_{\gamma_1, \omega_1}^{\circ,\heartsuit}\times\mathbf{M}_{\gamma_1, \omega_1}^{\circ,\heartsuit})/\mathbf{G}_{\gamma_1}\times\mathbf{G}_{\gamma_2}}="v2"; (10,0)*+{\mathbf{M}_{\hat{\gamma_{1}},\hat{\gamma_{2}}; \hat{\omega_1}, \hat{\omega_2}}^{\circ}/\mathbf{G}_{\hat{\gamma_1}}\times\mathbf{G}_{\hat{\gamma_2}}}="v3";
(80,0)*+{(\mathbf{M}_{\hat{\gamma_1}, \hat{\omega_1}}^{\circ}\times\mathbf{M}_{\hat{\gamma_1}, \hat{\omega_1}}^{\circ})/\mathbf{G}_{\hat{\gamma_1}}\times\mathbf{G}_{\hat{\gamma_2}}}="v4";
(-20,40)*+{C_1}="v5";
{\ar@{->}^{p_{\gamma_1,\gamma_2,\omega_1,\omega_2}^{\circ,\heartsuit}} "v1";"v2"}; {\ar@{->}^{p_{\hat{\gamma_1},\hat{\gamma_2},\hat{\omega_1},\hat{\omega_2}}^\circ} "v3";"v4"};
{\ar@{->}^{\epsilon_{1,2;1\times 2}}
"v1";"v3"}; {\ar@{->}^{\epsilon_{1\times 2}}_\wr "v2";"v4"}; {\ar@{->}^{u_1} "v1";"v5"}; {\ar@/^1pc/^{u'_1} "v5";"v2"}; {\ar@/_2pc/_{u''_1} "v5";"v3"}
\end{xy}$$ where $\epsilon_{1,2;1\times 2}$ is an affine fibration with fiber $\mathbb{C}^{\gamma_1^{i_-}\gamma_2^{i_-}}$ since $\mathbf{M}_{\hat{\gamma_{1}},\hat{\gamma_{2}}; \hat{\omega_1},\hat{\omega_2}}^{\circ}/\mathbf{G}_{\hat{\gamma_1}}\times\mathbf{G}_{\hat{\gamma_2}}\cong(\mathbf{M}_{\gamma_{1},\gamma_{2}; \omega_1,\omega_2}^{\circ,\heartsuit}/\GL_{\gamma_1^{i_-},\gamma_2^{i_-}}(\mathbb{C}))/\mathbf{G}_{\hat{\gamma_1}}\times\mathbf{G}_{\hat{\gamma_2}}$, and $u_1$ is an affine fibration with fiber $\mathbb{C}^{\gamma_2^{i_+}\gamma_1^{i_-}}=\mathbb{C}^{\gamma_2^{i_-}\gamma_1^{i_-}}$. Then $u_1^*$ is an isomorphism, which implies that $$u_{1,*}u_1^*=(u_1^*)^{-1}u_1^*u_{1,*}u_1^*=(u_1^*)^{-1}[\mathfrak{eu}(u_1)^{-1}]u_1^*=\mathfrak{eu}(u_1)^{-1},$$ where $\mathfrak{eu}(u_1)=\prod_{\alpha=1}^{\gamma_2^{i_+}}\prod_{\alpha'=1}^{\gamma_1^{i_-}}(x_{i_-,\alpha'}-x_{i_+,\alpha})$. The last equality holds because $u_1^*$ intertwines the actions of $\coh(C_1,\QQ)$ and $\coh(\mathbf{M}_{\gamma_{1},\gamma_{2}; \omega_1, \omega_2}^{\circ,\heartsuit}/\mathbf{G}_{\gamma_1}\times\mathbf{G}_{\gamma_2},\QQ)$.

In addition, it's easy to see that there exists $\nu$ satisfying the assumption of Theorem \ref{base}. Thus in the diagram
$$\begin{xy}
(0,20)*+{\coh_{c,\mathbf{G}_{\gamma_1}\times\mathbf{G}_{\gamma_2}}^{\bullet,crit}(\mathbf{M}_{\gamma_{1},\gamma_{2}}^{sp,\heartsuit},W_{\gamma_1,\gamma_2})^\vee\otimes\T^{d_1}}="v1";
(80,20)*+{\coh_{c,\mathbf{G}_{\gamma_1}\times\mathbf{G}_{\gamma_2}}^{\bullet,crit}(\mathbf{M}_{\gamma_1}^{sp,\heartsuit}\times\mathbf{M}_{\gamma_1}^{sp,\heartsuit},W_{\gamma_1}\boxplus W_{\gamma_2})^\vee}="v2";
(0,0)*+{\coh_{c,\mathbf{G}_{\hat{\gamma_1}}\times\mathbf{G}_{\hat{\gamma_2}}}^{\bullet,crit}(\mathbf{M}_{\hat{\gamma_{1}},\hat{\gamma_{2}}}^{sp},W_{\hat{\gamma_1},\hat{\gamma_2}})^\vee\otimes\T^{\hat{d}_1}}="v3";
(80,0)*+{\coh_{c,\mathbf{G}_{\hat{\gamma_1}}\times\mathbf{G}_{\hat{\gamma_2}}}^{\bullet,crit}(\mathbf{M}_{\hat{\gamma_1}}^{sp}\times\mathbf{M}_{\hat{\gamma_1}}^{sp},W_{\hat{\gamma_1}}\boxplus W_{\hat{\gamma_2}})^\vee}="v4";
{\ar@{->}_{\alpha^\heartsuit} "v2";"v1"};
{\ar@{->}_{\hat{\alpha}}
"v4";"v3"};{\ar@{->}^{\epsilon_{1,2;1\times 2,*}}
"v1";"v3"}; {\ar@{->}^{\epsilon_{1\times 2,*}} "v2";"v4"}
\end{xy}$$ we have $$\epsilon_{1,2;1\times 2,*}\alpha^\heartsuit=u''_{1,*}u_{1,*}u_1^*u_1^{'*}=u''_{1,*}[\mathfrak{eu}(u_1)^{-1}]u_1^{'*}=[\mathfrak{eu}(u_1)^{-1}]\hat{\alpha}\epsilon_{1\times 2,*}.$$
Here $\hat{d}_1=\sum_{a\in\widehat{\Omega}}\gamma_2^{s(a)}\gamma_1^{t(a)}$.

\item[$\bullet$]For $\beta$, the perimeter of the following diagram is Cartesian: $$\begin{xy}
(10,20)*+{\mathbf{M}_{\gamma_{1},\gamma_{2}; \omega_1,\omega_2}^{\circ,\heartsuit}/\mathbf{G}_{\gamma_1}\times\mathbf{G}_{\gamma_2}}="v1";
(80,20)*+{\mathbf{M}_{\gamma_1, \gamma_2; \omega_1,\omega_2}^{\circ,\heartsuit}/\mathbf{G}_{\gamma_1,\gamma_2}}="v2";
(10,0)*+{\mathbf{M}_{\hat{\gamma_{1}},\hat{\gamma_{2}}; \hat{\omega_1},\hat{\omega_2}}^{\circ}/\mathbf{G}_{\hat{\gamma_1}}\times\mathbf{G}_{\hat{\gamma_2}}}="v3";
(80,0)*+{\mathbf{M}_{\hat{\gamma_1}, \hat{\gamma_2}; \hat{\omega_1}, \hat{\omega_2}}^{\circ}/\mathbf{G}_{\hat{\gamma_1},\hat{\gamma_2}}}="v4"; (-20,40)*+{C_2}="v5"; {\ar@{->}^{q_{\gamma_{1},\gamma_{2},\omega_1,\omega_2}^{\circ\heartsuit}} "v1";"v2"};
{\ar@{->}^{q_{\hat{\gamma_{1}},\hat{\gamma_{2}},\hat{\omega_1},\hat{\omega_2}}^\circ}
"v3";"v4"};{\ar@{->}^{\epsilon_{1,2;1\times 2}}
"v1";"v3"}; {\ar@{->}^{\epsilon_{1,2}}_\wr "v2";"v4"};{\ar@{->}^{u_2} "v1";"v5"}; {\ar@/^1pc/^{u'_2} "v5";"v2"}; {\ar@/_2pc/_{u''_2} "v5";"v3"}
\end{xy}$$ and $u_2$ is an affine fibration with fiber $\mathbb{C}^{\gamma_2^{i_+}\gamma_1^{i_-}}=\mathbb{C}^{\gamma_2^{i_-}\gamma_1^{i_-}}$. Then similarly as above, $u_{2,*}u_2^*=\mathfrak{eu}(u_2)^{-1}$, and $\nu$ exists. Therefore in the following diagram
$$\begin{xy}
(0,20)*+{\coh_{c,\mathbf{G}_{\gamma_1}\times\mathbf{G}_{\gamma_2}}^{\bullet,crit}(\mathbf{M}_{\gamma_{1},\gamma_{2}}^{sp,\heartsuit},W_{\gamma_1,\gamma_2})^\vee\otimes\T^{d_0}}="v1";
(80,20)*+{\coh_{c,\mathbf{G}_{\gamma_1,\gamma_2}}^{\bullet,crit}(\mathbf{M}_{\gamma_1,\gamma_2}^{sp,\heartsuit},W_{\gamma_1,\gamma_2})^\vee}="v2";
(0,0)*+{\coh_{c,\mathbf{G}_{\hat{\gamma_1}}\times\mathbf{G}_{\hat{\gamma_2}}}^{\bullet,crit}(\mathbf{M}_{\hat{\gamma_1},\hat{\gamma_2}}^{sp},W_{\hat{\gamma_1},\hat{\gamma_2}})^\vee\otimes\T^{\hat{d}_0}}="v3";
(80,0)*+{\coh_{c,\mathbf{G}_{\hat{\gamma_1},\hat{\gamma_2}}}^{\bullet,crit}(\mathbf{M}_{\hat{\gamma_1},\hat{\gamma_2}}^{sp},W_{\hat{\gamma_1},\hat{\gamma_2}})^\vee}="v4";
{\ar@{->}_{\beta^\heartsuit} "v2";"v1"};
{\ar@{->}_{\hat{\beta}}
"v4";"v3"};{\ar@{->}^{\epsilon_{1,2;1\times 2,*}}
"v1";"v3"}; {\ar@{->}^{\epsilon_{1,2,*}} "v2";"v4"}
\end{xy}$$ we get $$\epsilon_{1,2;1\times 2,*}\beta^\heartsuit=u''_{2,*}u_{2,*}u_2^*u_2^{'*}=u''_{2,*}[\mathfrak{eu}(u_2)^{-1}]u_2^{'*}=[\mathfrak{eu}(u_2)^{-1}]\hat{\beta}\epsilon_{1,2,*}.$$ by Theorem \ref{base}. Here $\hat{d}_0=\sum_{i\in\hat{I}}\gamma_2^i\gamma_1^i$.

\item[$\bullet$]For $\zeta$, the following diagram commutes: $$\begin{xy}
(0,20)*+{\mathbf{M}_{\gamma_1, \gamma_2; \omega_1,\omega_2}^{\circ,\heartsuit}/\mathbf{G}_{\gamma_1,\gamma_2}}="v1";
(80,20)*+{\mathbf{M}_{\gamma,\omega}^{\circ,\heartsuit}/\mathbf{G}_{\gamma_1,\gamma_2}}="v2";
(0,0)*+{\mathbf{M}_{\hat{\gamma_1}, \hat{\gamma_2}; \hat{\omega_1}, \hat{\omega_2}}^{\circ}/\mathbf{G}_{\hat{\gamma_1},\hat{\gamma_2}}}="v3";
(80,0)*+{\mathbf{M}_{\hat{\gamma}, \hat{\omega}}^{\circ}/\mathbf{G}_{\hat{\gamma_1},\hat{\gamma_2}}}="v4";
{\ar@{->}^{j_{\gamma_{1},\gamma_{2},\omega_1,\omega_2}^{\circ,\heartsuit}} "v1";"v2"};
{\ar@{->}^{j_{\hat{\gamma_{1}},\hat{\gamma_{2}},\hat{\omega_1},\hat{\omega_2}}^\circ} "v3";"v4"};{\ar@{->}^{\epsilon_{1,2}}_\wr
"v1";"v3"}; {\ar@{->}^{\epsilon_{\gamma;1,2}} "v2";"v4"}
\end{xy}$$

thus the diagram
$$\begin{xy}
(0,20)*+{\coh_{c,\mathbf{G}_{\gamma_1,\gamma_2}}^{\bullet,crit}(\mathbf{M}_{\gamma_1,\gamma_2}^{sp,\heartsuit},W_{\gamma_1,\gamma_2})^\vee}="v1";
(80,20)*+{\coh_{c,\mathbf{G}_{\gamma_1,\gamma_2}}^{\bullet,crit}(\mathbf{M}_\gamma^{sp,\heartsuit},W_\gamma)^\vee}="v2";
(0,0)*+{\coh_{c,\mathbf{G}_{\hat{\gamma_1},\hat{\gamma_2}}}^{\bullet,crit}(\mathbf{M}_{\hat{\gamma_1},\hat{\gamma_2}}^{sp},W_{\hat{\gamma_1},\hat{\gamma_2}})^\vee}="v3";
(80,0)*+{\coh_{c,\mathbf{G}_{\hat{\gamma_1},\hat{\gamma_2}}}^{\bullet,crit}(\mathbf{M}_{\hat{\gamma}}^{sp},W_{\hat{\gamma}})^\vee}="v4";
{\ar@{->}^{\zeta^\heartsuit} "v1";"v2"};
{\ar@{->}^{\hat{\zeta}} "v3";"v4"};{\ar@{->}^{\epsilon_{1,2,*}}
"v1";"v3"}; {\ar@{->}^{\epsilon_{\gamma;1,2,*}} "v2";"v4"}
\end{xy}$$ commutes as well.

\item[$\bullet$]For $\delta$, the commutativity of the diagram $$\begin{xy}
(0,20)*+{\mathbf{M}_{\gamma,\omega}^{\circ,\heartsuit}/\mathbf{G}_{\gamma_1,\gamma_2}}="v1";
(80,20)*+{\mathbf{M}_{\gamma,\omega}^{\circ,\heartsuit}/\mathbf{G}_\gamma}="v2";
(0,0)*+{\mathbf{M}_{\hat{\gamma}, \hat{\omega}}^{\circ}/\mathbf{G}_{\hat{\gamma_1},\hat{\gamma_2}}}="v3";
(80,0)*+{\mathbf{M}_{\hat{\gamma}, \hat{\omega}}^{\circ}/\mathbf{G}_{\hat{\gamma}}}="v4";
{\ar@{->}^{pr_{\gamma_{1},\gamma_{2},\omega_1,\omega_2}^{\circ,\heartsuit}} "v1";"v2"};
{\ar@{->}^{pr_{\hat{\gamma_{1}},\hat{\gamma_{2}},\hat{\omega_1},\hat{\omega_2}}^\circ} "v3";"v4"};{\ar@{->}^{\epsilon_{\gamma;1,2}} "v1";"v3"}; {\ar@{->}^{\epsilon_\gamma}_\wr "v2";"v4"}
\end{xy}$$
gives rise to the commutativity of the diagram
$$\begin{xy}
(0,20)*+{\coh_{c,\mathbf{G}_{\gamma_1,\gamma_2}}^{\bullet,crit}(\mathbf{M}_\gamma^{sp,\heartsuit},W_\gamma)^\vee}="v1";
(80,20)*+{\coh_{c,\mathbf{G}_{\gamma}}^{\bullet,crit}(\mathbf{M}_\gamma^{sp,\heartsuit},W_\gamma)^\vee}="v2";
(0,0)*+{\coh_{c,\mathbf{G}_{\hat{\gamma_1},\hat{\gamma_2}}}^{\bullet,crit}(\mathbf{M}_{\hat{\gamma}}^{sp},W_{\hat{\gamma}})^\vee}="v3";
(80,0)*+{\coh_{c,\mathbf{G}_{\hat{\gamma}}}^{\bullet,crit}(\mathbf{M}_{\hat{\gamma}}^{sp},W_{\hat{\gamma}})^\vee}="v4";
{\ar@{->}^{\delta^\heartsuit} "v1";"v2"};
{\ar@{->}^{\hat{\delta}} "v3";"v4"};{\ar@{->}^{\epsilon_{\gamma;1,2,*}} "v1";"v3"}; {\ar@{->}^{\epsilon_{\gamma,*}} "v2";"v4"}
\end{xy}$$

\item[$\bullet$]Thom-Sebastiani isomorphism: by the pushforward of (\ref{TSsheaves}) and the natural isomorphism (\ref{PF}), we get the compatibility of the Thom-Sebastiani isomorphism with $\epsilon$.
\end{itemize}

Thus by the above proof of compatibility of $\epsilon$ and $\iota$ with the COHA multiplication, we have
\begin{thm}
The map $\mathfrak{c}:=\epsilon_*\iota^*$ is compatible with the COHA multiplication.
\end{thm}

When there is no potential, an explicit formula of the COHA multiplication is given in \cite[Sec.~2]{KoSo3}. It turns out that the COHA in this case is the shuffle algebra, and we have a more explicit description of $\mathfrak c$ in the next section. In particular, $\mathfrak c$ is surjective for Dynkin quivers.

The compatibility of $\iota$ with the COHA comultiplication can be shown using the following argument: the diagram involving $\overleftarrow{\alpha}$ satisfies Theorem \ref{base}; the diagram of $\overleftarrow{\beta}$ has only pullbacks up to $\mathfrak{eu}(q)$ since $\overleftarrow{\beta}=q_{\gamma_{1},\gamma_{2},\omega_1,\omega_2}^{\circ,*}\mathfrak{eu}(q)$, $\overleftarrow{\beta^\heartsuit}=q_{\gamma_{1},\gamma_{2},\omega_1,\omega_2}^{\circ,\heartsuit,*}\mathfrak{eu}(q^\heartsuit)$, and $\mathfrak{eu}(q)=\mathfrak{eu}(q^\heartsuit)$; the diagrams involving $\overleftarrow{\zeta}$ and $\overleftarrow{\delta}$ both consist of only pullbacks. For $\epsilon$, the diagram involving $\overleftarrow{\alpha}$ consists of only pushforwards; by the arguments on $\beta$ we get $\epsilon_{1,2;1\times 2,*}\beta^\heartsuit=[\mathfrak{eu}(u_2)^{-1}]\hat{\beta}\epsilon_{1,2,*}$, and since $\mathfrak{eu}(q_{\gamma_1,\gamma_2,\omega_1,\omega_2}^{\circ,\heartsuit})=\mathfrak{eu}(q_{\hat{\gamma_1},\hat{\gamma_2},\hat{\omega_1},\hat{\omega_2}}^\circ)\mathfrak{eu}(u_2)$, we have $\epsilon_{1,2;1\times 2,*}\overleftarrow{\beta^\heartsuit}=\epsilon_{1,2;1\times 2,*}\beta^\heartsuit\mathfrak{eu}(q_{\gamma_1,\gamma_2,\omega_1,\omega_2}^{\circ,\heartsuit})=[\mathfrak{eu}(u_2)^{-1}]\hat{\beta}\epsilon_{1,2,*}\mathfrak{eu}(q_{\gamma_1,\gamma_2,\omega_1,\omega_2}^{\circ,\heartsuit})=\hat{\beta}\epsilon_{1,2,*}\mathfrak{eu}(q_{\hat{\gamma_1},\hat{\gamma_2},\hat{\omega_1},\hat{\omega_2}}^\circ)=\overleftarrow{\hat{\beta}}\epsilon_{1,2,*}$.The perimeter of the concatenation of the diagrams involving $\overleftarrow{\zeta}$ and $\overleftarrow{\delta}$ $$\begin{xy}
(0,20)*+{\mathbf{M}_{\gamma_1, \gamma_2; \omega_1,\omega_2}^{\circ,\heartsuit}/\mathbf{G}_{\gamma_1,\gamma_2}}="v1";
(50,20)*+{\mathbf{M}_{\gamma,\omega}^{\circ,\heartsuit}/\mathbf{G}_{\gamma_1,\gamma_2}}="v2";
(100,20)*+{\mathbf{M}_{\gamma,\omega}^{\circ,\heartsuit}/\mathbf{G}_\gamma}="v3";
(0,0)*+{\mathbf{M}_{\hat{\gamma_1}, \hat{\gamma_2}; \hat{\omega_1}, \hat{\omega_2}}^{\circ}/\mathbf{G}_{\hat{\gamma_1},\hat{\gamma_2}}}="v4";
(50,0)*+{\mathbf{M}_{\hat{\gamma}, \hat{\omega}}^{\circ}/\mathbf{G}_{\hat{\gamma_1},\hat{\gamma_2}}}="v5";
(100,0)*+{\mathbf{M}_{\hat{\gamma}, \hat{\omega}}^{\circ}/\mathbf{G}_{\hat{\gamma}}}="v6"; {\ar@{->}^{j_{\gamma_{1},\gamma_{2},\omega_1,\omega_2}^{\circ,\heartsuit}} "v1";"v2"};
{\ar@{->}^{j_{\hat{\gamma_{1}},\hat{\gamma_{2}},\hat{\omega_1},\hat{\omega_2}}^\circ} "v4";"v5"};{\ar@{->}^{\epsilon_{1,2}}_\wr "v1";"v4"}; 
{\ar@{->}^{pr_{\gamma_{1},\gamma_{2},\omega_1,\omega_2}^{\circ,\heartsuit}} "v2";"v3"};
{\ar@{->}^{pr_{\hat{\gamma_{1}},\hat{\gamma_{2}},\hat{\omega_1},\hat{\omega_2}}^\circ} "v5";"v6"};{\ar@{->}^{\epsilon_{\gamma;1,2}} "v2";"v5"}; {\ar@{->}^{\epsilon_\gamma}_\wr "v3";"v6"}
\end{xy}$$ satisfies the Theorem \ref{base}, thus $\epsilon_{1,2,*}\overleftarrow{\zeta^\heartsuit}\overleftarrow{\delta^\heartsuit}=\overleftarrow{\hat{\zeta}}\overleftarrow{\hat{\delta}}\epsilon_{\gamma,*}$.

To conclude, we have the following theorem by the above argument on the compatibility of $\epsilon$ and $\iota$ with the COHA comultiplication:
\begin{thm}
The map $\mathfrak{c}$ is compatible with the COHA comultiplication.
\end{thm}

\subsection{Spherical subalgebras}

We study the effects of the edge contraction on the spherical subalgebra on the shuffle algebra. Let $\mathbf H^{=,s}_Q=\mathbf H^s\cap\mathbf H^=_Q$, where $\mathbf H^=_Q=\oplus_{\gamma^{i_+}=\gamma^{i_-}}\mathbf H_{Q,\gamma}$, and $\mathbf H^s_{Q,=}$ be the subalgebra of $\mathbf H^{=,s}_Q$ generated by $\mathbf H_{e_i},\,i\in I\setminus\{i_+,\,i_-\}$ and $\mathbf H_{e_{i_+}}\mathbf H_{e_{i_-}}$, $\mathbf H_{e_{i_-}}\mathbf H_{e_{i_+}}$.

\begin{prop}
There is an algebra homomorphism
\begin{equation}
\mathbf H^s_{Q,=}\rightarrow\mathbf H^s_{\widehat{Q}}\nonumber
\end{equation}
given by quotient $x_{i_-}-x_{i_+}$.
\end{prop}

In general $\mathbf H^{=,s}_Q$ is not mapped to $\mathbf H^s_{\widehat{Q}}$. Consider the cyclic quiver $C_n$ consisting of n vertices $\{1,\ldots,n\}$ and arrows $a_i:i\rightarrow i+1,\;i\in\Zn_n$, and fix $i_+\in\{1,\ldots,n\}$ and let $i_-=i_++1$. Then $C_{n-1}$ is the quiver obtained from $C_n$ by edge contraction along $i_+\rightarrow i_-$. Then the computation using the explicit formula of the shuffle multiplication (\ref{shufflemult}) shows that $\mathbf H^{=,s}_{C_n}$ is not mapped to $\mathbf H^s_{C_{n-1}}$: in the fac of the edge contraction of an element in $\mathbf H^{=,s}_{C_n,\gamma}$ obtained via multiplication $\mathbf H^{=,s}_{C_n,e_{i_-}}\otimes\mathbf H^{=,s}_{C_n,e_{i_-+1}}\otimes\cdots\otimes\mathbf H^{=,s}_{C_n,e_{i_+}}\rightarrow\mathbf H^{=,s}_{C_n,\gamma}$ for $\gamma=(1,\ldots,1)$ , there are at least $n-1$ factors of the form $x_i-x_{i-1}$, but an element in $\mathbf H^s_{C_{n-1},\hat{\gamma}}$ has less than $n-1$ factors.

\subsection{Compatibility of Drinfeld double}

We investigate the compatibility of the edge contraction 
\begin{equation}
\begin{array}{ll}
\mathfrak{c}: &\mathbf{H}_Q^=\rightarrow\mathbf{H}_{\hat{Q}},
\\&f(x_{i,\alpha},x_{i_+,\alpha},x_{i_-,\alpha})\mapsto f(x_{i,\alpha},x_{i_0,\alpha}),\quad where\;i\in I\setminus\{i_+,i_-\}
\end{array}\nonumber
\end{equation}
with the Hopf structures and the Drinfeld double of the COHA. Here we use the same notation for the edge contraction map of COHA and the shuffle algebra, and $\mathbf H_Q^==\oplus_{\gamma^{i_+}=\gamma^{i_-}}\mathbf{H}_\gamma$. Let $H^{0,\heartsuit}$ be the subalgebra of $H^0$ generated by $\{\psi_i, i\in I\setminus\{i_+,i_-\}\}$ and $\psi_{i_+}\psi_{i_-}$, then it acts on $\mathbf{H}_Q^=$. The edge contraction can be extended to $\mathbf{H}^{\mathfrak{e}}$ via $\psi_i\mapsto\psi_i,i\in I\setminus\{i_+,i_-\}$ and $\psi_{i_+}\psi_{i_-}\mapsto\psi_{i_0}$.

First, for the compatibility of the $H^{0,\heartsuit}$-action and the $\hat{H^0}$-action, let's compare $(\psi_{i_+}\psi_{i_-})f(\psi_{i_+}\psi_{i_-})^{-1}$ and $\psi_{i_0}\mathfrak{c}(f)\psi_{i_0}^{-1}$: $$(\psi_{i_+}\psi_{i_-})f(\psi_{i_+}\psi_{i_-})^{-1}=f\frac{fac^+_1\cdot fac^-_1}{fac^+_2\cdot fac^-_2},$$ where $$fac^+_1=\frac{1}{\Pi^{\gamma^{i_+}}_{\alpha=1}(x_{i_+,\alpha}-z)}\prod_{r_+}\prod_{\alpha=1}^{\gamma^{i_+}}(x_{i_+,\alpha}-z)\prod_{i_+\rightarrow i_-}\prod_{\alpha=1}^{\gamma^{i_-}}(x_{i_-,\alpha}-z)\prod_{\substack{i_+\rightarrow i\\i\neq i_\pm}}\prod_{\alpha=1}^{\gamma^i}(x_{i,\alpha}-z),$$ $$fac^+_2=\frac{1}{\Pi^{\gamma^{i_+}}_{\alpha=1}(z-x_{i_+,\alpha})}\prod_{r_+}\prod_{\alpha=1}^{\gamma^{i_+}}(z-x_{i_+,\alpha})\prod_{i_-\rightarrow i_+}\prod_{\alpha=1}^{\gamma^{i_-}}(z-x_{i_-,\alpha})\prod_{\substack{i\rightarrow i_+\\i\neq i_\pm}}\prod_{\alpha=1}^{\gamma^i}(z-x_{i,\alpha}),$$ $$fac^-_1=\frac{1}{\Pi^{\gamma^{i_-}}_{\alpha=1}(x_{i_-,\alpha}-z)}\prod_{r_-}\prod_{\alpha=1}^{\gamma^{i_-}}(x_{i_-,\alpha}-z)\prod_{i_-\rightarrow i_+}\prod_{\alpha=1}^{\gamma^{i_+}}(x_{i_+,\alpha}-z)\prod_{\substack{i_-\rightarrow i\\i\neq i_\pm}}\prod_{\alpha=1}^{\gamma^i}(x_{i,\alpha}-z),$$ and $$fac^-_2=\frac{1}{\Pi^{\gamma^{i_-}}_{\alpha=1}(z-x_{i_-,\alpha})}\prod_{r_-}\prod_{\alpha=1}^{\gamma^{i_-}}(z-x_{i_-,\alpha})\prod_{i_+\rightarrow i_-}\prod_{\alpha=1}^{\gamma^{i_+}}(z-x_{i_+,\alpha})\prod_{\substack{i\rightarrow i_-\\i\neq i_\pm}}\prod_{\alpha=1}^{\gamma^i}(z-x_{i,\alpha}),$$ here $r_\pm$ are the numbers of loops at $i_\pm$.

On the other hand, $$\psi_{i_0}\mathfrak{c}(f)\psi_{i_0}^{-1}=f\frac{fac^0_1}{fac^0_2},$$ where $$fac^0_1=\frac{1}{\Pi^{\gamma^{i_0}}_{\alpha=1}(x_{i_0,\alpha}-z)}\prod_{r_0}\prod_{\alpha=1}^{\gamma^{i_0}}(x_{i_0,\alpha}-z)\prod_{\substack{i_0\rightarrow i\\i\neq i_0}}\prod_{\alpha=1}^{\gamma^i}(x_{i,\alpha}-z),$$
and $$fac^0_2=\frac{1}{\Pi^{\gamma^{i_0}}_{\alpha=1}(z-x_{i_0,\alpha})}\prod_{r_0}\prod_{\alpha=1}^{\gamma^{i_0}}(z-x_{i_0,\alpha})\prod_{\substack{i\rightarrow i_0\\i\neq i_0}}\prod_{\alpha=1}^{\gamma^i}(z-x_{i,\alpha}),$$ here $r_0$ is the number of loops at $i_0$.

Since
\begin{equation}\label{edges}
\begin{array}{ll}
r_0=r_++r_-+|\{i_+\rightarrow i_-,i_-\rightarrow i_+\}\setminus\{a_0\}|, \\\{i_0\rightarrow i, i\neq i_0\}=\{i_+\rightarrow i,i_-\rightarrow i, i\neq i_\pm\}, \\\{i\rightarrow i_0, i\neq i_0\}=\{i\rightarrow i_+, i\rightarrow i_-, i\neq i_\pm\},
\end{array}
\end{equation}
the edge contraction is compatible with the $H^0$-action on $\mathbf{H}$.

Same argument using (\ref{shufflemult}) and (\ref{edges}) implies that the edge contraction is compatible with the multiplication on $\mathbf{H}^\mathfrak{e}$.

Second, for the comultiplication on $\mathbf{H}^\mathfrak{e}$, we just need to check the elements in $\mathbf{H}_{(1,1),Q}$ and $\mathbf{H}_{1,\widehat{Q}}$. Let $f(x_{[1,(1,1)]})=f(x_{i_+},x_{i_-})\in\mathbf{H}_{(1,1),Q}$, then
\begin{equation}\nonumber
\begin{array}{ll}
\Delta(f)=&\psi_{i_+}(x_{i_+})\psi_{i_-}(x_{i_-})\otimes f(x_{i_+},x_{i_-})\\&+\frac{\psi_{i_+}(x_{i_+})f(x_{i_-}\otimes x_{i_+})}{fac(x_{i_+}|x_{i_-})}+\frac{\psi_{i_-}(x_{i_-})f(x_{i_+}\otimes x_{i_-})}{fac(x_{i_-}|x_{i_+})}\\&+f(x_{i_+},x_{i_-})\otimes1.
\end{array}
\end{equation}
After restricting to $\mathbf{H}_{\gamma^{i_+}=\gamma^{i_-}}^{\mathfrak{e}}$, we get $$\Delta(f)=\psi_{i_+}(x_{i_+})\psi_{i_-}(x_{i_-})\otimes f(x_{i_+},x_{i_-})+f(x_{i_+},x_{i_-})\otimes1.$$ For $g(x_{i_0})\in\mathbf{H}_{1,\widehat{Q}}$, we have $\widehat{\Delta}(g)=\psi_{i_0}(x_{i_0})\otimes g(x_{i_0})+g(x_{i_0})\otimes1$. Moreover, $\mathfrak{c}(\Delta(\psi_{i_+}\psi_{i_-}))=\mathfrak{c}(\Delta(\psi_{i_+})\Delta(\psi_{i_-}))=\mathfrak{c}((\psi_{i_+}\otimes\psi_{i_+})(\psi_{i_-}\otimes\psi_{i_-}))=\mathfrak{c}((\psi_{i_+}\psi_{i_-})\otimes(\psi_{i_+}\psi_{i_-}))=\psi_{i_0}\otimes\psi_{i_0}=\widehat{\Delta}(\psi_{i_0})$. Thus the edge contraction is compatible with the comultiplication.

Third, to show that the edge contraction is compatible with the antipode $S$, we only need to show that $(-1)^{|\gamma|}=(-1)^{|\hat{\gamma}|}$. This holds since $\gamma^{i_+}=\gamma^{i_-}$, which implies that $|\gamma|=\prod_{i\in I}\gamma^i!=\prod_{i\in I\setminus\{i_+,i_-\}}\gamma^i!\cdot\gamma^{i_+}!\cdot\gamma^{i_-}!$ and $|\hat{\gamma}|=\prod_{i\in\hat{I}}\gamma^i!$ have the same parity.

All the above arguments work for $\mathbf{H}^{\mathfrak{e},coop}$ if the edge contraction is extended via $\phi_i\mapsto\phi_i,i\in I\setminus\{i_+,i_-\}$ and $\phi_{i_+}\phi_{i_-}\mapsto\phi_{i_0}$.

Fourth, for the bilinear paring $(\cdot,\cdot)$, the proof (on both $H^0$ and $\mathbf{H}$) is similar as the one for the $H^0$-action on $\mathbf{H}$, except the factor $|A|!$. The change from $|A|!$ to $|\hat{A}|!$ comes from combining the terms involving $x_{i_-,\alpha}-x_{i_-,\alpha'}$ for $\alpha\neq\alpha'$, and there are $|\gamma^{i_-}|!$ such terms.

Finally, we prove the following

\begin{thm}
The edge contraction $\mathfrak{c}$ induces an algebra homomorphism of the Drinfeld doubles $D\mathfrak c:D(\Hc_{Q,W}^{=,s})\rightarrow D(\Hc_{\widehat Q,\widehat W}^s)$.
\end{thm}
\begin{proof}
It suffices to prove that $D\mathfrak{c}((1\otimes g)(f\otimes1))=D\mathfrak{c}(1\otimes g)D\mathfrak{c}(f\otimes1)$ for $f(x_{i_+},x_{i_-})\in\mathbf{H}^{\mathfrak{e}}_{(1,1)}$ and $g(x_{i_+},x_{i_-})\in\mathbf{H}^{\mathfrak{e},coop}_{(1,1)}$. We have \begin{equation}
\begin{array}{ll}
\Delta^2(f)=&\psi_{i_+}(x_{i_+})\psi_{i_-}(x_{i_-})\otimes\psi_{i_+}(x_{i_+})\psi_{i_-}(x_{i_-})\otimes f(x_{i_+},x_{i_-})\\&+\psi_{i_+}(x_{i_+})\psi_{i_-}(x_{i_-})\otimes f(x_{i_+},x_{i_-})\otimes1+f(x_{i_+},x_{i_-})\otimes1\otimes1,
\end{array}\nonumber
\end{equation}
and
\begin{equation}
\begin{array}{ll}
(\Delta^{op})^2(g)=&g(x_{i_+},x_{i_-})\otimes\phi_{i_+}(x_{i_+})\phi_{i_-}(x_{i_-})\otimes\phi_{i_+}(x_{i_+})\phi_{i_-}(x_{i_-})\\&+1\otimes g(x_{i_+},x_{i_-})\otimes\phi_{i_+}(x_{i_+})\phi_{i_-}(x_{i_-})+1\otimes1\otimes g(x_{i_+},x_{i_-}).
\end{array}\nonumber
\end{equation}
Thus
\begin{equation}
\begin{array}{ll}
(1\otimes g)(f\otimes1)=&(\psi_{i_+}(x_{i_+})\psi_{i_-}(x_{i_-}),1)\psi_{i_+}(x_{i_+})\psi_{i_-}(x_{i_-})\otimes1(f,g)\\&+(\psi_{i_+}(x_{i_+})\psi_{i_-}(x_{i_-}),1)f\otimes g(1,\phi_{i_+}(x_{i_+})\phi_{i_-}(x_{i_-}))\\&+(f,-g\phi_{i_+}(x_{i_+})^{-1}\phi_{i_-}(x_{i_-})^{-1})1\\&\otimes\phi_{i_+}(x_{i_+})\phi_{i_-}(x_{i_-})(1,\phi_{i_+}(x_{i_+})\phi_{i_-}(x_{i_-})),
\end{array}\nonumber
\end{equation}
which is clearly preserved by the edge contraction since each ingredient is preserved by the edge contraction.
\end{proof}

\subsection{Scattering diagram and DT-series}

The edge contraction induces an embedding of the scattering diagramms associated with $(\widehat Q,\widehat W)$ and $(Q,W)$, and a map between their DT-series.

\begin{thm}
Let $Q$ be a quiver such that there's no other arrows from $i_+$ to $i_-$ except $a_0$, and no loop at $i_-$. For any wall $\hat{\mathfrak d}$ of $\widehat{\mathfrak D}$, these exists $k\in\mathbb R$ such that $\eta: \hat{\kappa}=(\hat{\kappa}_i, i\in I\setminus\{+,-\},\hat\kappa_0)\mapsto\kappa=(\kappa_i=\hat{\kappa}_i, i\in I\setminus\{+,-\}, k\kappa_+=\kappa_-, \kappa_++\kappa_-=\hat{\kappa}_0)$ embeds $\hat{\mathfrak d}$ to a wall $\mathfrak d$ in $\mathfrak D$. Moreover, one can choose one $k$ that works for all walls assuming that the number of walls in $\widehat{\mathfrak D}$ is finite, and $\eta$ is linear. In addation, the wall-crossing automorphisms are compatible under $\eta$. Thus $\widehat{\mathfrak D}$ can be embedded in $\mathfrak D$.
\end{thm}

\begin{proof}
For any $\hat{\mathfrak d}\in\widehat{\mathfrak D}$ there exists a $\hat\kappa-$semistable $Jac(\widehat Q,\widehat W)$-representation $\widehat V$ for all $\hat\kappa\in\hat{\mathfrak d}$, i.e., $\hat\kappa(\widehat V)=0$ and $\hat\kappa(\widehat R)\leqslant0$ for all subrepresentations $\widehat R\subset\widehat V$. Equivalently, $\kappa(V)=\sum_{i\in I\setminus\{+,-\}}\kappa_i\gamma^i+\kappa_+\gamma^++\kappa_-\gamma^-=\sum_{i\in I\setminus\{+,-\}}\kappa_i\gamma^i+(\kappa_++\kappa_-)\gamma^0=\sum_{i\in I\setminus\{+,-\}}\hat\kappa_i\gamma^i+\hat\kappa_0\gamma^0=\hat\kappa(\widehat V)=0$, and similarly $\kappa(R)\leqslant0$ since $\dim R_+=\dim R_-$. We want $V$ to be $\kappa$-semistable, namely, any subrepresentation $U\subset V$ satisfies $\kappa(U)\leqslant0$, where $\dim U_+\leqslant\dim U_-$. The $R$'s above are subrepresentations of $V$ as $(Q,W)-$representations, so we only need to check the case where $\dim U_+<\dim U_-$. We construct a $R$ from $U$ as follows. We cut $R'_-\subset U_-$ such that $R'_-\oplus R_-=U_-$ and $a_0(U_+)=R_-$. In order to get a subrepresentation we need to further cut the inverse images of $R_-'$ under all paths pointing to $i_-$. The whole cutted part is denoted by $U'$. Fix a $\hat\kappa\in\widehat{\mathfrak D}$ and choose $\kappa_-^0$ such that $\kappa^0(U')=0$. Then $\kappa^0(U)=\kappa^0(R)=\hat{\kappa}^0(\widehat W)=\hat\kappa(\widehat R)\leqslant0$, thus $\kappa(U)\leqslant\kappa^0(U)\leqslant0$ for any $\kappa$ whose $\kappa_-\leqslant\kappa_-^0$. There are finitely many subrepresentations $U\subset V$, thus finitely many $\kappa^0_-$, so these exists a $\kappa$ such that $\kappa(U)\leqslant\kappa^0(U)\leqslant0$ for all $\kappa^0$. Since $\hat{\mathfrak d}$ is a cone, there are $\hat m_1^{\hat{\mathfrak d}},\ldots,\hat m_p^{\hat{\mathfrak d}}\in\widehat M_{\mathbb R}$ such that $\hat{\mathfrak d}=\{\sum c_l^{\hat{\mathfrak d}}\hat m_l^{\hat{\mathfrak d}}\mid c_l^{\hat{\mathfrak d}}\in\mathbb R_{\geqslant0}\}$. For any $\hat m_l^{\hat{\mathfrak d}}$ there's a $\kappa_{l,-}^{\hat{\mathfrak d},0}$ as above. 
We can find a $k$ such that $\kappa_{l,+}^{\hat{\mathfrak d}}+\kappa_{l,-}^{\hat{\mathfrak d}}=\hat m_{l,0}^{\hat{\mathfrak d}}$, $\kappa_{l,-}^{\hat{\mathfrak d}}=k\kappa_{l,+}^{\hat{\mathfrak d}}$, and $\kappa_{l,-}^{\hat{\mathfrak d}}\leqslant\kappa_{l,-}^{\hat{\mathfrak d},0}$ for any $l=1,\ldots,p$, namely, $\eta(\hat m_l^{\hat{\mathfrak d}})=\kappa_l^{\hat{\mathfrak d}}\in\mathfrak d$. Since $\eta$ is linear, $\eta(\sum c_l^{\hat{\mathfrak d}}\hat m_l^{\hat{\mathfrak d}})=\sum c_l^{\hat{\mathfrak d}}\eta(\hat m_l^{\hat{\mathfrak d}})\in\mathfrak d$ because $\mathfrak d$ is a cone. Thus there exists a $k$ such that $\eta(\hat{\mathfrak d})\subset\mathfrak d\subset M_{\mathbb R}$ for any $\hat{\mathfrak d}\in\widehat{\mathfrak D}$ since $\widehat{\mathfrak D}$ has finitely walls. Therefore the underlying cone complex of $\hat{\mathfrak D}$ is embedded in $\mathfrak D$.

By the same argument, the wall-crossing automorphisms associated to the corresponding walls $\hat{\mathfrak d}$ and $\mathfrak d$ are compatible. A sufficiently general path $\hat p\subset\widehat M_{\mathbb R}$ (with end points not in the intersections of $\eta(\widehat{\mathfrak D})$ and any $\mathfrak d\in\mathfrak D$) gives rise to a (sufficiently general) path $p\subset M_{\mathbb R}$ with the same endpoints. Different choices of $p$ gives the same wall-crossing automorphisms $F(p)$. Then $F(p)$ is compatlbe with $F(\hat p)$ if we map the automorphisms given by representations with $\dim V_+\neq\dim V_-$ to 0.
\end{proof}


\subsection{Compatibility with mutations}

The mutations of quivers with potential is introduced in \cite{DWZ}. Given a quiver with potential $(Q,W)$, the mutation at $i\in Q_0$, denoted by $\mu_i$, is an operation to get another quiver with potential $\mu_i(Q,W)$. There is an intermediate step $\tilde{\mu}_i$ in the construction of $\mu_i(Q,W)$. Please refer to \cite[Sec.~5]{DWZ} for details.

The mutations are expected to closely related to braid group actions. In this section we study how the mutations $\mu_{i_+}$ and $\mu_{i_-}$ of $(Q,W)$ are related to the mutation $\mu_{i_0}$ of $(\widehat Q, \widehat W)$. In order to perform mutations, we assume that the quivers $Q$ and $\widehat Q$ have neither loops, nor two cycles involving $i_+$, $i_-$ and $i_0$. Because $\widehat Q$ has no loops, $a_0$ is the only arrow connecting $i_+$ and $i_-$. Since $i_0$ should not belong any two cycles, there's no cycle of length three involving $i_+$. 

\begin{thm}
Let $(Q,W)$ and $(\widehat Q,\widehat W)$ satisfy the above assumptions. If in addition $i_+$ is not the source of any other arrows besides $a_0$, then $\mu_{i_+}\mu_{i_-}\mu_{i_+}$ coincides with $\mu_{i_0}$ under the edge contraction. If $i_-$ is not the target of any other arrows, then $\mu_{i_-}\mu_{i_+}\mu_{i_-}$ coincides with $\mu_{i_0}$ under the edge contraction.
\end{thm}

\begin{proof}
We prove the first case, i.e., $i_+$ is not the source of any other arrows besides $a_0$. The proof of the second case is very similar. First let's compute $\mu_{i_+}\mu_{i_-}\mu_{i_+}(Q,W)$. After mutating at $i_+$, we get a quiver $\mu_{i_+}(Q)$ with $a_0$ and $a: i\rightarrow i_+$ reversed, i.e., they are replaced by $a_0^*: i_-\rightarrow i_+$ and $a^*: i_+\rightarrow i$ respectively; moreover, one adds new arrows $[a_0a]: i\rightarrow i_-$ for any $a: i\rightarrow i_+$. The potential $\mu_{i_+}(W)=[W]+\sum_{t(a)=i_+}[a_0a]a^*a_0^*$, where $[W]$ is obtained by replacing $a_0a$ in $W$ with $[a_0a]$. Then we mutate $\mu_{i_+}(Q,W)$ at $i_-$. This gives us a quiver $\mu_{i_-}\mu_{i_+}(Q)$ with all the arrows connecting $i_-$ reversed. In particular, $a_0^*$ becomes $a_0$ and $[a_0a]$ becomes $[a_0a]^*$. Moreover, new arrows $[a_0^*b]$ and $[cb]$ are added for any $b: i\rightarrow i_-$ and $c: i_-\rightarrow i$ where $i\neq i_+$, and the trivial part is deleted. The potential $\mu_{i_-}\mu_{i_+}(W)=([[W]]+[cb]b^*c^*)_{red}+\sum_{t(b)=i_-}[a_0^*b]b^*a_0$, where $[[W]]$ is obtained by replacing $a_0^*b$ in $[W]$ with $[a_0^*b]$, and $([[W]]+[cb]b^*c^*)_{red}$ is obtained by deleting the trivial part. Note that the arrows $a^*: i_+\rightarrow i$ and $[a_0^*[a_0a]]: i\rightarrow i_+$ lead to a trivial part of $\tilde{\mu}_{i_-}\mu_{i_+}(Q,W)$, so they are deleted to get ${\mu_{i_-}\mu_{i_+}}(Q,W)$. The last step is to mutate ${\mu_{i_-}\mu_{i_+}}(Q,W)$ at $i_+$. The resulted quiver has $a_0^*$ and $[a_0^*b]^*$ the only arrows connecting $i_+$, and $b^*$ and $a_0[a_0^*b]$ are deleted. Namely, the arrows $a: i\rightarrow i_+$ in $Q$ become $[a_0a]^*: i_-\rightarrow i$ in $\mu_{i_+}\mu_{i_-}\mu_{i_+}(Q)$, and $b: i\rightarrow i_-$ become $[a_0^*b]^*: i_+\rightarrow i$ in $\mu_{i_+}\mu_{i_-}\mu_{i_+}(Q)$. The potential $\mu_{i_+}\mu_{i_-}\mu_{i_+}(W)=[[W]]_{red}$. Thus by comparing $\mu_{i_+}\mu_{i_-}\mu_{i_+}(Q,W)$ obtained above with $\mu_{i_0}(\widehat Q, \widehat W)$, we can see that they coincide under the edge contraction.
\end{proof}

\subsection{Compatibility with dimensional reduction}\label{cdr}

Fix a quiver $Q$ and construct its double quiver $\overline{Q}$, the preprojective algebra $\Pi_Q$ and the triple quiver with potential $(\widetilde{Q},W)$ as in Section \ref{DimRed}. We contract $(\widetilde{Q},W)$ along $a_0\in\Omega$ and show that it gives the edge contraction of $\Pi_Q$.

Denote $(\widehat{\widetilde{Q}},\widehat{W})$ the quiver with potential after contraction along $a_0$, then $$\widehat{W}=\sum_{a\in\Omega,a\neq a_0}\hat{l}[\hat{a},\hat{a^*}]+\sum_{\substack{a\in\Omega,s(a)=i_- or\\ t(a)=i_-,a\neq a_0}}\hat{l}_{i_-}[\hat{a},\hat{a^*}]+\hat{l}_{i_-}\hat{a_0^*}-l_{i_+}\hat{a_0^*},$$
where $\hat{l}=\sum_{i\neq i_-}l_i$. One can show that $\widehat{W}=W$ via $a\mapsto\hat{a}$, $a^*\mapsto\hat{a^*}$ and $l_i\mapsto\hat{l}_i$, since $\hat{l}_{i_-}\hat{a_0^*}-l_{i_+}\hat{a_0^*}=a_0^{-1}l_{i_-}a_0a_0^*a_0-l_{i_+}a_0^*a_0=l_{i_-}a_0a_0^*-l_{i_+}a_0^*a_0=l_{i_0}[a_0,a_0^*]$.

Take $\{\hat{l}_i\}=\{l_i,i\neq i_-,\hat{l}_{i_-}\}$ to be a cut, then 
\begin{equation}\label{pd1}
\frac{\partial\widehat{W}}{\partial l_i}=\sum_{t(a)=i,a\neq a_0}\hat{a}\hat{a^*}-\sum_{s(a)=i,a\neq a_0}\hat{a^*}\hat{a}, \quad i\neq i_+\;or\;i_-
\end{equation}
\begin{equation}\label{pd2}
\frac{\partial\widehat{W}}{\partial l_{i_+}}=\sum_{t(a)=i_+,a\neq a_0}\hat{a}\hat{a^*}-\sum_{s(a)=i_+,a\neq a_0}\hat{a^*}\hat{a}-\hat{a_0^*},
\end{equation}
and \begin{equation}\label{pd3}
\frac{\partial\widehat{W}}{\partial \hat{l}_{i_-}}=\sum_{t(a)=i_-,a\neq a_0}\hat{a}\hat{a^*}-\sum_{s(a)=i_-,a\neq a_0}\hat{a^*}\hat{a}+\hat{a_0^*}.
\end{equation}
Eliminating $\hat{a_0^*}$ by adding up (\ref{pd2}) and (\ref{pd3}), and together with (\ref{pd1}), we get ADHM for $\overline{\widehat Q}$, namely, all the relations in $\Pi_{\widehat Q}$.

To conclude, we have the following commutative diagram:

$$\begin{xy}
(0,20)*+{\coh_{c,\mathbf{G}_\gamma}^{\bullet,crit}(\mathbf{M}_{\widetilde{Q},\gamma}^{sp},W_\gamma)^\vee}="v1";
(80,20)*+{\coh_{c,\mathbf{G}_{\hat{\gamma}}}^{\bullet,crit}(\mathbf{M}_{\widehat{\widetilde{Q}},\hat{\gamma}}^{sp},\widehat{W}_{\hat{\gamma}})^\vee}="v2";
(0,0)*+{\coh_{c,\mathbf{G}}^{\bullet}(\mathbf{M}_{\Pi_Q,\gamma}^{sp},\QQ)^{\vee}\otimes\T^{-\gamma\cdot\gamma}}="v3";
(80,0)*+{\coh_{c,\mathbf{G}_{\hat{\gamma}}}^{\bullet}(\mathbf{M}_{\Pi_{\widehat Q},\hat{\gamma}}^{sp},\QQ)^{\vee}\otimes\T^{-\hat{\gamma}\cdot\hat{\gamma}}}="v4";
{\ar@{->}^{\mathfrak{c}} "v1";"v2"};
{\ar@{->}^{\mathfrak{c}} "v3";"v4"};{\ar@{->}^{\wr} "v1";"v3"}; {\ar@{->}^{\wr} "v2";"v4"}
\end{xy}$$



\subsection{2d edge contraction}

We will study how the edge contraction $\mathfrak{c}$ will affect the preprojective algebra $\Pi_Q$ of a quiver $Q$. The result is the same as the one induced by dimensional reduction in section \ref{cdr}. For any arrow $a\in\overline{\Omega}$, set $\hat{a}=aa_0$ if $t(a_0)=s(a)$, $\hat{a}=a_0^{-1}a$ if $t(a_0)=t(a)$, and $\hat{a}_{i_-}=a_0^{-1}a_{i_-}a_0$ if $t(a)=s(a)=i_-$, and $\hat{a}=a$ otherwise. The ADHM relation defining $\Pi_Q$ is $$\sum_{a\in\Omega}[a,a^*]=\sum_{\substack{s(a)\neq i_-,\\t(a)\neq i_-}}[a,a^*]+\sum_{\substack{s(a)=i_-, or\\t(a)=i_-,a\neq a_0}}[a,a^*]+[a_0,a_0^*]=0.$$
This is equivalent to
\begin{equation}\label{ADHM1}
\sum_{\substack{t(a)=i,\\a\neq a_0}}aa^*-\sum_{\substack{s(a)=i,\\a\neq a_0}}a^*a=0,\quad i\neq i_+\;or\;i_-,
\end{equation}
\begin{equation}\label{ADHM2}
\sum_{\substack{t(a)=i_+,\\a\neq a_0}}aa^*-\sum_{\substack{s(a)=i_+,\\a\neq a_0}}a^*a-a_0^*a_0=0,
\end{equation}
and
\begin{equation}\label{ADHM3}
\sum_{\substack{t(a)=i_-,\\a\neq a_0}}aa^*-\sum_{\substack{s(a)=i_-,\\a\neq a_0}}a^*a+a_0^*a_0=0,
\end{equation}

Note that (\ref{ADHM1}) is equivalent to
\begin{equation}\label{ecADHM1}
\sum_{\substack{t(a)=i,\\a\neq a_0}}aa_0a_0^{-1}a^*-\sum_{\substack{s(a)=i,\\a\neq a_0}}a^*a_0a_0^{-1}a=\sum_{\substack{t(a)=i,\\a\neq a_0}}\hat{a}\hat{a^*}-\sum_{\substack{s(a)=i,\\a\neq a_0}}\hat{a^*}\hat{a}=0,\quad i\neq i_+\;or\;i_-=0,
\end{equation}
and $a_0^{-1}(\ref{ADHM3})a_0$ gives us
\begin{equation}\label{ecADHM3}
\sum_{\substack{t(a)=i_-,\\a\neq a_0}}\hat{a}\hat{a^*}-\sum_{\substack{s(a)=i_-,\\a\neq a_0}}\hat{a^*}\hat{a}+a_0^*a_0=0,
\end{equation}

We eliminate $a_0^*$ by adding up (\ref{ADHM2}) and (\ref{ecADHM3}), and then together with (\ref{ADHM1}) we get the ADHM relation of $\widehat{Q}$. Thus the edge contraction of $Q$ induces the edge contraction of the corresponding preprojective algebras.


Addresses:

Y. L.: Department of Mathematics, SUNY Buffalo, Buffalo, NY 14260, USA, {yiqiang@buffalo.edu}

J. R.: Department of Mathematics, SUNY Buffalo, Buffalo, NY 14260, USA, {jren5@buffalo.edu}

\end{document}